\documentclass[11pt, reqno]{amsart}
\usepackage{latexsym, amsmath, amssymb, amsthm, mathtools}
\usepackage{cases, verbatim}
\usepackage[active]{srcltx}
\usepackage{esint}
\usepackage[colorlinks]{hyperref}
\usepackage{hypernat}
\usepackage{xcolor}
\usepackage{float}
\usepackage[normalem]{ulem}
\usepackage{cancel}
\usepackage{subcaption}
\usepackage[T1]{fontenc}
\usepackage{mathscinet}
\hypersetup{
    unicode=false,          
    pdftoolbar=true,        
    pdfmenubar=true,        
    pdffitwindow=false,     
    colorlinks=false,       
    linkcolor=red,          
    linkbordercolor=red,
    citecolor=green,        
    citebordercolor=green,
    filecolor=magenta,      
    urlcolor=cyan,           
   urlbordercolor={1 1 1}  
}

\usepackage[margin=1.5in]{geometry}

\newtheorem{thm}{Theorem}
\newtheorem{lem}[thm]{Lemma}
\newtheorem{prop}[thm]{Proposition}

\theoremstyle{remark}
\newtheorem{rmk}[thm]{Remark}
\newtheorem{example}[thm]{Example}

\theoremstyle{definition}
\newtheorem{defi}[thm]{Definition}

\numberwithin{thm}{section} 
\numberwithin{equation}{section}

\makeatletter

\newcommand{\Rmnum}[1]{\expandafter\@slowromancap\romannumeral #1@}
\makeatother

\def\Om{\Omega}

\def\R{{\mathbb R}}
\def\L{{\mathbb L}}
\def\X{{\mathbf X}}

\def\U{{\mathcal U}}
\def\S{{\mathcal S}}
\def\Z{{\mathbb Z}}

\newcommand{\Oba}{\overline{\Omega}}

\newcommand{\vep}{\varepsilon}
\newcommand{\ol}{\overline}
\newcommand{\ul}{\underline}

\newcommand{\bpm}{\begin{pmatrix}}
\newcommand{\epm}{\end{pmatrix}}

\newcommand{\beq}{\begin{equation}}
\newcommand{\eeq}{\end{equation}}

\def\lipl{{\rm Lip_{loc}\,}}

plus 6pt minus 12pt
\title[Stability and Perron's method for Monge solutions]{Stability and Perron's method for Monge solutions of eikonal equations in metric spaces}

\author[Q.~Liu]{Qing Liu}
\address[Qing Liu]{Geometric Partial Differential Equations Unit, Okinawa Institute of Science and Technology Graduate University, Okinawa 904-0495, Japan}
\email{qing.liu@oist.jp}

\author[M.~B.~P.~Wiranata]{Made Benny Prasetya Wiranata}
\address[Made Benny Prasetya Wiranata]{Geometric Partial Differential Equations Unit, Okinawa Institute of Science and Technology Graduate University, Okinawa 904-0495, Japan, and \newline
Faculty of Mathematics and Natural Sciences, Universitas Gadjah Mada, Indonesia}
\email{made.wiranata@oist.jp}

\date{\today}

\begin{document}

\begin{abstract}
The main result of this paper concerns the stability of Monge solutions to eikonal equations in metric spaces. Although the subslope used in the standard definition of Monge solutions is, in general, not stable under uniform convergence, we overcome this difficulty by introducing an equivalent nonlocal characterization based on sphere-wise subslope conditions. This approach complements existing methods based on metric viscosity solutions and enables us to establish stability directly under uniform convergence and Gromov-Hausdorff perturbations of the domain and metric, without the use of test functions. The method extends to inhomogeneous equations via the optical metric and to a broader class of Hamilton-Jacobi equations. Our stability argument further enables us to develop Perron’s method for Monge solutions in general complete length spaces. This yields a new Perron-type existence result for Hamilton-Jacobi equations in metric spaces, without relying on test functions. 
\end{abstract}


\maketitle


\section{Introduction}

\subsection{Background and motivation}
Hamilton-Jacobi equations on metric spaces arise naturally in optimal control, geometric optics, and the analysis of singular geometric structures. Among them, the eikonal equation
\begin{equation}\label{eikonal}
    |\nabla u|=f\quad\text{in }\Omega,
\end{equation}
plays a fundamental role. Here, $(\X,d)$ is a complete length space, $\Omega\subsetneq\X$ is a domain, and $f:\Omega\to(0,\infty)$ is a given continuous function. Since a general metric space lacks a differentiable structure, the meaning of $|\nabla u|$ must be interpreted appropriately in order to formulate well-posed boundary value problems.

Several notions of weak solutions have been proposed for Hamilton-Jacobi equations in metric spaces. Giga, Hamamuki, and Nakayasu~\cite{GHN} introduced a notion of metric viscosity solutions by reducing the equation to one-dimensional problems along arc-length parametrized curves. This approach was extended to the time-dependent setting in \cite{Na1} and applied to eikonal equations on the Sierpi\'nski gasket in \cite{CCM}. A different viscosity framework was developed by Ambrosio and Feng~\cite{AF} and by Gangbo and {\'S}wi{\polhk{e}}ch~\cite{GaS2,GaS}, extending the classical Euclidean definition through suitable classes of test functions. A third approach, called {Monge solutions}, was introduced by Liu, Shanmugalingam, and Zhou~\cite{LShZ}, where it was shown that, for the eikonal equation \eqref{eikonal} on complete length spaces, the Monge formulation is equivalent to the two viscosity notions mentioned above. However, unlike viscosity solutions, Monge solutions are defined directly through the metric geometry and do not rely on test functions. More precisely, the metric gradient $|\nabla u|$ is interpreted as the local subslope
\begin{equation}\label{sub_slope}
    |\nabla^-u|(x)=\limsup_{\Omega\ni y\to x}\frac{(u(x)-u(y))_+}{d(x,y)},
\end{equation}
where $(a)_+=\max\{a,0\}$. The equation \eqref{eikonal} is then understood simply as
\[
|\nabla^-u|(x)=f(x)
\qquad\text{for every }x\in\Omega.
\]
This formulation provides a remarkably simple and geometric description of solutions while remaining equivalent to the existing viscosity theory. Applications of the local subslope to metric determination have recently been developed in \cite{DaLeSa,DaSa}.

The notion of Monge solutions was proposed by Newcomb and Su~\cite{NeSu} in the Euclidean setting as an alternative approach to Hamilton-Jacobi equations, particularly suitable for discontinuous Hamiltonians. A similar idea of characterizing viscosity solutions pointwise without using test functions appeared earlier in the work of Lions and Souganidis~\cite{LiS}. We refer to
\cite{NeSu,So,CaSi1,BrDa}
for related developments in Euclidean spaces, to \cite{EGV} for recent work in sub-Riemannian geometry, to \cite{LW} for time-dependent Hamilton-Jacobi equations in metric spaces,
and to \cite{LMit1} for applications to the eigenvalue problem of the infinity Laplacian.

Despite the simple formulation of Monge solutions, their stability theory is more delicate. Suppose that $f_j\to f$ uniformly in $\Omega$ and that $u_j$ is a Monge solution of
\[
|\nabla u_j|=f_j
\quad\text{in }\Omega.
\]
Even if $u_j$ converges uniformly to a function $u$, it is not immediate that $u$ remains a Monge solution of
\eqref{eikonal}. The main difficulty is that the local subslope \eqref{sub_slope} is generally not stable under uniform perturbations of the function. On the other hand, for the approach of Giga, Hamamuki, and Nakayasu, we refer to \cite[Section 5]{GHN} for stability results concerning their notion of solutions. More recently, stability for the solutions introduced in \cite{GaS2, GaS} has been established by Nakayasu and Namba~\cite{NN}, Makida~\cite{M}, and Makida and Nakayasu~\cite{MN} in compact or proper geodesic spaces. These results can be viewed as a generalization of the classical stability theory for viscosity solutions (see, e.g., \cite{CIL, BCBook}). By the equivalence among these notions of solutions established in \cite{LShZ}, one may further obtain stability for Monge solutions of eikonal equations. However, it is natural to ask whether a direct stability argument for Monge solutions is possible. 

\subsection{Key observation and main result}
The main purpose of this paper is to address this stability question based on a characterization of Monge solutions by a nonlocal slope. We restrict ourselves to stationary equations in this work, but a similar approach can be applied to the time-dependent case, based on the notion of Monge solutions introduced in \cite{LW}. Let us begin with the standard eikonal equation
\begin{equation}\label{eikonal1}
    |\nabla u|=1\quad\text{in }\Omega,
\end{equation}
for which the Monge condition simply takes the form
\begin{equation}\label{1-slope1}
    |\nabla^-u|(x)=1\quad \text{for all $x\in\Omega$.}
\end{equation}
Motivated by the dynamic programming principle associated to its control-theoretic interpretation, we introduce, for every radius $r>0$ such that $\overline{B_r(x)}\subset\Omega$,
\begin{equation}\label{r-subslope}
    S_r^-u(x)=\sup_{y\in\partial B_r(x)}
    \frac{(u(x)-u(y))_+}{r},
\end{equation}
Here, we denote ${B_r(x)}=\{y\in \X: d(x, y)< r\}$ 
and use $\ol{B_r(x)}$ and $\partial B_r(x)$ to denote the corresponding closed ball and sphere in $\X$. An important observation, at least in the Euclidean case, is that \eqref{1-slope1} is equivalent to 
\begin{equation}\label{1-slope2}
    S_r^-u(x)=1\quad
    \text{for every $r>0$ such that $\overline{B_r(x)}\subset\Omega$.}
\end{equation}
This characterization of Monge solutions is closely related to the equivalence result in \cite[Theorem~1.2]{LShZ} and, more recently, \cite[Lemma~3.5]{JC} in general complete length spaces. We provide a more complete characterization in Theorem~\ref{thm:char} concerning not only Monge solutions but also Monge sub- and supersolutions. This result  plays a central role in our analysis. 

The advantage of using \eqref{1-slope2} lies in its natural stability under uniform convergence. Indeed, for every fixed $x$ and $r$, the quantity $S_r^-u(x)$ depends only on the values of $u$ on the sphere $\partial B_r(x)$ and therefore passes directly to uniform limits. As a consequence, the stability of Monge solutions can be established without appealing to any equivalent viscosity formulation. The same characterization also provides the foundation for other applications developed in this paper, including stability under perturbations of the underlying metric and domain, a homogenization example for eikonal equations on periodically weighted lattice graphs, and Perron's method for Monge solutions.

Our approach is inspired by the stability argument developed in \cite[Section 6]{BrDa} for Hamilton-Jacobi equations in Euclidean spaces, which relies on optical length functions arising from the optimal control interpretation. It essentially builds upon the well-known equivalence between viscosity solutions and the dynamic programming principle established by Lions and Souganidis~\cite{LiS}; see also \cite[Theorem 2.32]{BCBook}. Our main contribution in this work is to extend this perspective to general geometric settings. In particular, we establish a local stability theory that allows the underlying domains themselves to vary, rather than only the Hamiltonian or the ambient metric. Stability problems of this type have already been considered in Makida~\cite{M} and Makida and Nakayasu~\cite{MN} using the approach of \cite{GaS2, GaS}. In addition to the difference in the notion of solutions, in this work, we improve their results by removing unnecessary compactness or properness assumptions on the ambient and approximating spaces. 

Moreover, a key distinction from \cite{M, MN} is that we allow perturbations of the domain $\Omega$, rather than the entire space $\X$. As noted in \cite[Remark~5.10]{LN} and \cite[Section~4.2]{LShZ2}, Hamilton-Jacobi equations posed on a complete metric space $\X$ can often be interpreted as a state-constraint or generalized Dirichlet boundary condition in the Euclidean setting, which may essentially restrict the class of solutions. Let us explain this point through the following simple example. Consider the eikonal equation $|\nabla u|=1$ in $\X=[0, 1]\subset \R$, which is clearly a complete metric space under the standard Euclidean metric. It is easily seen that $u(x)=x$ is a viscosity (or Monge) solution in the interior $(0, 1)$ but not on the whole space $[0, 1]$, since the solution property fails at $x=0$. In this work, we aim to develop a local stability theory for the equation that is essentially free from boundary effects, potentially allowing for a broader range of applications. 

Our first main stability result is Theorem~\ref{thm:dom-stability} below. In the sequel, we denote by $d_H(A, B)$  an equivalent form of the Hausdorff distance of sets $A, B\subset (\X, d)$, defined by
\[
d_H(A, B)=\sup_{x\in A} \inf_{y\in B} d(x, y)+ \sup_{y\in B}\inf_{x\in A} d(x, y). 
\]
Moreover, for a general metric $L$, we use $|\nabla_L f|$ to denote the slope of $f$ with respect to $L$. 

\begin{thm}[Stability of Monge solutions]\label{thm:dom-stability}
    Let $(\X,d)$ be a metric space, and let $(\X_\vep,d_\vep)\subset(\X,d)$ be complete length spaces for $\vep\ge0$. Assume that $(\X_0,d_0)$ is proper and 
    \begin{equation}\label{compat-metric}
        d_0(x_j,x)\to0\quad \text{implies} \quad  d(x_j,x)\to0 \quad \text{for $x_j, x\in\X_0$.}
    \end{equation}
    Let $\Omega_\vep\subsetneq\X_\vep$ be domains in $(\X_\vep, d_\vep)$ for $\vep\geq 0$ satisfying the following conditions. 
    \begin{enumerate}
        \item[(A1)] If $x_\vep, y_\vep\in \Omega_\vep$, $x_0, y_0\in \Omega_0$  satisfy
        \begin{equation}\label{dom-convergence1}
            d(x_\vep, x_0)\to 0,\ d(y_\vep, y_0)\to 0 \quad \text{as $\vep\to 0$,}
        \end{equation}
        then
        \begin{equation}\label{dom-convergence2}
            d_\vep(x_\vep, y_\vep)\to d_0(x_0, y_0) \quad \text{as $\vep\to 0$.}
        \end{equation}
        \item[(A2)] For any $x_0\in \Omega_0$ and $r_0>0$ small satisfying $V_0:=\{x\in \X_0: d_0(x, x_0)\leq r_0\}\subset \Omega_0$, there exist $x_\vep\in \Omega_\vep$, $r_\vep>0$ such that 
        \begin{equation}\label{dom-ball}
            V_\vep:=\{x\in \X_\vep: d_\vep(x, x_\vep)\leq r_\vep\}\subset \Omega_\vep,
        \end{equation}
        and $r_\vep\to r_0$, $d(x_\vep, x_0)\to 0$, $d_H(V_\vep, V_0)\to 0$ as $\vep\to 0$. 
    \end{enumerate}
    Let $u_\vep\in \lipl(\Omega_\vep)$ be a Monge solution of $|\nabla_{d_\vep} u_\vep|=1$ in $\Omega_\vep$ for each $\vep>0$. Assume that $u_\vep\to u$ as $\vep\to 0$ in the sense that 
    \begin{equation}\label{unif-converge}
        |u_\vep(x_\vep)- u(x_0)|\to 0 \quad \text{whenever $x_0\in \Omega_0$, $x_\vep\in \Omega_\vep$ and $d(x_\vep, x_0)\to 0$.}
    \end{equation}
    Then $u$ is a Monge solution of $|\nabla_{d_0} u|=1$ in $\Omega_0$. 
\end{thm}

The assumption (A1) is a natural condition for the stability result. Indeed, for any $\vep\geq 0$ and $y_\vep\in \X_\vep$, $d_\vep(\cdot, y_\vep)$ is a Monge solution of $|\nabla u_\vep|=1$ in $\Omega_\vep\setminus \{y_\vep\}$, and therefore the convergence \eqref{dom-convergence2} under the condition \eqref{dom-convergence1} is precisely the stability property we aim to show. An assumption similar to (A1) is also imposed in \cite{MN} for a different approach to stability of more general Hamilton-Jacobi equations. 

The condition (A2) is new and proves to be necessary as well to ensure stability under domain perturbations. A simple example illustrating instability in the absence of (A2) is as follows.
\begin{example}
Let $\X=\X_\vep=\R^2$ with $d, d_\vep$ given by the Euclidean metric for all $\vep\geq 0$. Let $\Omega_\vep=B_1(0)\setminus \ol{B_\vep(0)}$. It is clear that $\Omega_\vep\to \Omega_0=B_1(0)$ as $\vep\to 0$ in the Hausdorff sense. However, while $u_\vep(x)=|x|$ is a Monge solution of $|\nabla u_\vep|=1$ in $\Omega_\vep$ for all $\vep>0$, its uniform limit $u_0(x)=|x|$ (for $x\in \Omega_0$) is not a Monge solution of the equation in $\Omega_0$, since the subslope associated to $d_0$ at $x=0$ is  $|\nabla^- u_0|(0)=0\neq 1$. In this case, any closed ball $\ol{B_r(0)}$ with $r\in (0, 1)$ contained in $\Omega_0$ cannot be approximated from inside by closed balls in $\Omega_\vep$, as required by \eqref{dom-ball} in (A2). On the other hand, if the Hausdorff limit of $\Omega_\vep$ is interpreted as $\Omega_0=B_1(0)\setminus \{0\}$, then (A2) is satisfied, and the stability result holds. 
\end{example}

As mentioned earlier, Theorem~\ref{thm:dom-stability} improves the results in \cite{M, MN}, as it no longer requires properness of the ambient space $(\X,d)$ or any of the approximating spaces $(\X_\vep,d_\vep)$. The only compactness needed in the supersolution argument is on the limit space $\X_0$. In general, removing all properness assumptions on the spaces may lead to instability of Monge solutions, as shown by the following example.

\begin{example}\label{ex:L2_instability}
Let us consider a fixed metric space $\X=\ell^2$, the space of square-summable sequences equipped with the norm metric. We also take a fixed domain $\Omega=B_r(0)\subset \ell^2$ with $r\in (0, 1)$. For $n\in \mathbb{N}$, let $e_n$ be the standard orthonormal basis and define 
\[
u_n(x):=\|x-e_n\|_{\ell^2}, \quad x\in \Omega.
\]
Since $e_n\notin\Omega$, each $u_n$ is a Monge solution of $|\nabla u_n|=1$ in $\Omega$. Moreover, for $x\in \Omega$, whenever $x_n\to x$ in $\ell^2$ as $n\to \infty$, we have 
\[
u_n(x_n)^2=\|x_n\|_{\ell^2}^2+1-2\langle x_n,e_n\rangle_{\ell^2}
\to \|x\|^2+1,
\]
because $\langle x_n,e_n\rangle_{\ell^2}\to0$. Thus, as $n\to \infty$, $u_n$ converges to
\[
u(x)=\sqrt{\|x\|_{\ell^2}^2+1}, \quad x\in \Omega
\]
in the sense of \eqref{unif-converge}. However, the limit $u$ is not a Monge solution; note that $u$ has a strict minimum at $0$, and hence $|\nabla^-u|(0)=0$. Here, conditions (A1) and (A2) hold trivially, but $\X$, as the limit space, is not proper. This example shows that for such general stability results, compactness assumptions, especially on the limit space, cannot be completely removed. 
 
In this example, $u_n$ does not converge to $u$ uniformly as $n\to \infty$. This shows that, in general, the condition \eqref{unif-converge} considered in this work is weaker than uniform convergence in infinite-dimensional spaces and therefore may not guarantee stability of solutions without additional space assumptions. If a sequence of Monge solutions $u_n$ of \eqref{eikonal1} does converge uniformly to some $u\in \lipl(\Omega)$ in $\Omega$, then $u$ is necessarily a Monge solution of \eqref{eikonal1} as well. This can also be deduced from the characterization \eqref{1-slope2}. See \cite[Proposition~3.3]{JC} for a stability result of this type in general length spaces.  
\end{example}

\subsection{Generalizations and a homogenization example}
The idea of adopting this nonlocal subslope can be extended to treat \eqref{eikonal} with a general inhomogeneous term $f: \X\to (0, \infty)$. In this case, one can change the metric $d$ to the so-called optical length function $L_f$ associated with the given $f$, defined by 
\begin{equation}\label{opt-length}
    L_f(x, y)=\inf_{\gamma\in \Gamma(x, y)} \int_\gamma f\, ds, \quad x, y\in \X,
\end{equation}
where $\Gamma(x, y)$ is the set of all rectifiable curves in $\X$ joining $x, y$ with finite path integral, i.e., 
\begin{equation}\label{curve-set}
    \Gamma(x, y):= \left\{\gamma: [0, \ell]\to \X: \ \gamma(0)=x, \gamma(\ell)=y, \ |\gamma'|=1 \ \text{for a.e. $s\in [0, \ell]$}\right\}. 
\end{equation}
Here, we assume $f$ to be bounded, continuous and $\inf_\X f>0$ in $\X$, which guarantees the finiteness of $L_f(x, y)$ for each pair of $x, y\in \X$. Under the new metric $L_f$, Monge solutions of the equation \eqref{eikonal} solve $|\nabla^-_{L_f} u|=1$ in $\Omega$,
where $|\nabla^-_{L_f}u|$ denotes the subslope with respect to $L_f$, namely, 
\[
|\nabla^-_{L_f} u|(x)=\limsup_{\Omega\ni y\to x}\dfrac{(u(x)-u(y))_+}{L_f(x,y)}, \quad \text{for $x\in \Omega$.}
\]
Then, we easily obtain a stability result for \eqref{eikonal} by adapting Theorem~\ref{thm:dom-stability} to the general setting; see Theorem~\ref{thm:dom-stability-gen}. We refer to \cite{LShZ2} for more details on the metric change method for \eqref{eikonal} in general metric measure spaces with a possibly discontinuous inhomogeneous term $f$. Our stability results can also be extended to the case of discontinuous $f$. 

Our method extends to a broader class of Hamilton-Jacobi equations 
\begin{equation}\label{general eq}
    H(x, u, |\nabla u|)=0 \quad \text{in $\Omega$,}
\end{equation}
where $H: \Oba\times \R\times [0, \infty)\to \R$ is continuous, strictly increasing in the last argument, and satisfies other structural assumptions. The strict monotonicity assumption allows a local reduction of \eqref{general eq} to an eikonal equation, so that the preceding stability arguments apply.

To illustrate our stability results, we present in Section~\ref{sec:example} a concrete example closely related to the homogenization of eikonal equations on periodically weighted square lattices approximating the Euclidean plane. For a positive $\mathbb Z^2$-periodic density $f\in C(\R^2)$, we consider the eikonal equations $|\nabla_{d_\varepsilon}u_\varepsilon|=f_\varepsilon$
on the rescaled lattices $\X_\varepsilon$, where $f_\varepsilon(x)=f(x/\varepsilon)$. We show that the associated optical metrics converge to the effective metric
\[
\bar d(x,y)=a_1|x_1-y_1|+a_2|x_2-y_2|,
\]
where $a_1$ and $a_2$ are the averaged optical costs in the horizontal and vertical directions. Consequently, our stability theorem yields that any limit of Monge solutions in the sense of \eqref{unif-converge} solves the effective equation $|\nabla_{\bar d}u|=1$; see Proposition~\ref{prop:lattice-example} for the precise statement. 

When $a_1=a_2$, this reduces to the classical eikonal equation with respect to the Euclidean $\|\cdot\|_1$ metric. In contrast, if $a_1\neq a_2$, the limiting Hamiltonian is anisotropic and, although it admits an explicit expression in Euclidean coordinates, it cannot be written in the form $|\nabla_d u|=f_0$ for any continuous function $f_0$. This example shows how our framework extends classical homogenization theory for Hamilton-Jacobi equations to periodic geodesic metric spaces with microscopically varying domains.

\subsection{Perron's method for Monge solutions}
As a further application of the characterization \eqref{1-slope2}, we develop Perron’s method in general metric spaces and recover the existence of Monge solutions to \eqref{eikonal1}. Existence of solutions to the time-dependent Hamilton-Jacobi equation based on the approach of Gangbo and {\'S}wi{\polhk{e}}ch~ was developed in \cite[Theorem 7.6]{GaS}, where the authors obtained a metric viscosity solution by taking the pointwise supremum over metric viscosity subsolutions in length spaces. On the other hand, the existence of Monge solutions is typically derived from the control-theoretic interpretation, and a direct Perron-type argument for Monge solutions in general metric spaces has remained less explored. Recently, \cite[Proposition~3.1]{JC} showed that the pointwise infimum of a family of Monge solutions of the simplest eikonal equation \eqref{eikonal1} is again a Monge solution. It remains to discuss the corresponding questions for pointwise suprema of Monge subsolutions and pointwise infima of Monge supersolutions to more general equations. Our results address these questions and, in particular, lead to the following Perron-type existence theorem for \eqref{eikonal}.

\begin{thm}[Perron's method for eikonal equations]\label{perron:gen_eikonal}
    Suppose that $(\X, d)$ is a complete length space and $\Omega\subsetneq \X$ is a domain. Assume that $f\in C(\Omega)$ is nonnegative. Assume that there exist a Monge subsolution $u_-$ and a Monge supersolution $u_+$ of \eqref{eikonal} such that $u_-\leq u_+$ in $\Omega$. Define $\ol{w},\ul{w}:\Omega\to \R$ by    \begin{equation}\label{Perron_sol_max}
        \ol{w}(x):=\sup\{u(x): \text{$u$ is a Monge subsolution of \eqref{eikonal} satisfying $u_-\leq u\leq u_+$ in $\Omega$}\},
    \end{equation}    \begin{equation}\label{Perron_sol_min}
        \ul{w}(x):=\inf\{u(x): \text{$u$ is a Monge supersolution of \eqref{eikonal} satisfying $u_-\leq u\leq u_+$ in $\Omega$}\}.
    \end{equation}
    Then, $\ol{w}$ and $\ul{w}$ are Monge solutions of \eqref{eikonal} with $\ul{w}\leq \ol{w}$.
\end{thm}

A more general version of Theorem~\ref{perron:gen_eikonal}, applicable to \eqref{general eq}, is given in Theorem~\ref{gen_perron's method} under monotonicity and coercivity conditions on $H$. Our approach generalizes the classical Perron-type arguments to the metric setting. In contrast to the viscosity framework, we directly verify that $w$ is a Monge solution, without using test functions. A key step, closely connected to our preceding stability analysis, is to prove that the class of Monge subsolutions is closed under pointwise supremum. This structural property enables the Perron construction, and its proof is again based on the characterization \eqref{1-slope2}. Further details are provided in Section~\ref{sec:perron}.

This paper is organized as follows. In Section~\ref{sec:monge}, we revisit the definition of Monge solutions in metric spaces and establish the characterization \eqref{1-slope2}. Section~\ref{sec:stability} is devoted to the proof of our main stability result, Theorem~\ref{thm:dom-stability}, together with a discussion of extensions to more general equations. In Section~\ref{sec:perron}, we study the existence of Monge solutions by developing Perron's method in general metric spaces. Our approach applies not only to the eikonal equation but also to a broader class of Hamilton--Jacobi equations.

\subsection*{Acknowledgments}
We would like to thank Nicolas~Dirr, Federica~Dragoni, and Xiaodan~Zhou for their valuable comments during the preparation of this paper. We are also grateful to Claudio~Marchi for helpful discussions during the early stages of this work. 

\subsection*{Data statement}
 Data sharing is not applicable to this article as no data sets were generated or analysed during the current study.


\section{Sphere-based characterization of Monge solutions}\label{sec:monge}

Let us first recall the notion of Monge solutions introduced in \cite{LShZ} for the eikonal equation \eqref{eikonal}. The definition of Monge solutions for a general class of Hamilton-Jacobi equations in length spaces is given in \cite[Definition 4.1]{LShZ}. In the case of \eqref{eikonal}, Monge solutions are defined as follows. 

\begin{defi}[Monge solutions]\label{defi:monge}
    A function $u\in \lipl(\Omega)$ is called a Monge subsolution (resp., Monge supersolution) of \eqref{eikonal} if, for every $x\in \Omega$,
    \[ 
    |\nabla^-u|(x)\leq f(x)\qquad (\text{resp.,}\ |\nabla^-u|(x)\geq f(x)). 
    \]
    A function $u\in \lipl(\Omega)$ is said to be a Monge solution of \eqref{eikonal} if $u$ is both a Monge subsolution and a Monge supersolution of \eqref{eikonal}.
\end{defi}

Our stability argument relies on the following characterization of Monge solutions to $|\nabla u|=1$ via spheres. We recall that $S_r^-u(x)$ is defined by \eqref{r-subslope} when $\ol{B_r(x)}\subset \Omega$ for a domain $\Omega\subsetneq \X$. 

\begin{thm}[Sphere-based characterization of Monge solutions]\label{thm:char}
    Suppose that $(\X, d)$ is a complete length space and $\Omega\subsetneq \X$ is a domain. Let $u\in \lipl(\Omega)$.  Then, the following assertions hold. 
    \begin{itemize}
        \item[(i)] $u$ is a Monge subsolution of \eqref{eikonal1} if and only if $S_r^-u(x)\leq 1$ for all $x\in \Omega$ and all $r>0$ such that $\ol{B_r(x)}\subset \Omega$.
        \item[(ii)] When $u$ is assumed to be bounded from below on each closed ball in $\Omega$, $u$ is a Monge supersolution of \eqref{eikonal1} if and only if $S_r^-u(x)\geq 1$ for all $x\in \Omega$ and all $r>0$ such that $\ol{B_r(x)}\subset \Omega$. 
        \item[(iii)] $u$ is a Monge solution of \eqref{eikonal1} if and only if $S_r^- u(x)=1$ for all $x\in \Omega$ and all $r>0$ such that $\ol{B_r(x)}\subset \Omega$.
    \end{itemize}
\end{thm}

\begin{proof}
Let us begin with the proof of (i). We first prove $\Leftarrow$. Fix $x_0\in \Omega$ and let $\delta>0$ such that $\ol{B_\delta(x_0)}\subset \Omega$. Then, $\partial B_r(x_0)\subset \Omega$ for all $0<r<\delta$. By assumption, we have 
\[
S_r^-u(x_0)=\sup_{y\in \partial B_r(x_0)}\dfrac{(u(x_0)-u(y))_+}{r}\leq1
\]
for all $0<r<\delta$. It follows immediately that $(u(x_0)-u(y))_+\leq d(x_0, y)$ for all $y\in B_\delta(x_0)\setminus \{x_0\}$, which amounts to saying
\begin{equation}\label{char-eq1}
    \sup_{y\in B_\delta(x_0)\setminus \{x_0\}}\dfrac{(u(x_0)-u(y))_+}{d(x_0, y)}\leq 1. 
\end{equation}
Sending $\delta\to 0$ in \eqref{char-eq1}, we obtain $|\nabla^- u|(x_0)\leq 1$. 
    
We next turn to the proof of $\Rightarrow$ for subsolutions. Fix $x_0\in \Omega$ again and choose $r>0$ satisfying $\ol{B_r(x_0)}\subset \Omega$. We claim that  
\begin{equation}\label{char-eq4}
    u(x_0)-u(y)\leq d(x_0, y)=r\quad \text{for any $y\in \partial B_r(x_0)$.}
\end{equation}
In fact, suppose, to the contrary, that $u(x_0)-u(y)>c r$ for some $c>1$ and $y\in\partial B_r(x_0)$, we can find an arc-length parametrized curve $\gamma: [0, \ell]\to \Omega$ such that $\gamma(0)=x_0, \gamma(\ell)=y$ and 
\begin{equation}\label{char-eq5}
    u(\gamma(0))-u(\gamma(\ell))\geq c\ell. 
\end{equation}
Let 
\[
\phi(t):=u(\gamma(t))+\frac{u(x_0)-u(y)}{\ell}t, \quad t\in [0, \ell].
\]
We thus see that $\phi(0)=\phi(\ell)$ and therefore $\phi$ attains a maximum over $[0, \ell]$ at some $\hat{t}\in [0, \ell)$. Denote $\hat{x}=\gamma(\hat{t})$. It follows that 
\[
u(\hat{x})=u(\gamma(\hat{t}))\geq u(\gamma(t))+ \frac{u(x_0)-u(y)}{\ell}(t-\hat{t})
\]
for all $t> \hat{t}$. In view of \eqref{char-eq5}, we obtain 
\[
|\nabla^- u|(\hat{x})\geq \limsup_{t\to \hat{t}+} \frac{u(\gamma(\hat{t}))-u(\gamma(t))}{t-\hat{t}}\geq \frac{u(x_0)-u(y)}{\ell}\geq c>1, 
\]
which is a contradiction to the assumption that $u$ is a Monge subsolution of \eqref{eikonal1}. We thus complete the proof of \eqref{char-eq4}, which immediately implies $S_r^- u(x_0)\leq 1$ as desired. This concludes the proof of the equivalence for subsolutions. 

Let us prove the assertion (ii) for Monge supersolutions. The implication $\Leftarrow$ can be handled in a similar manner to the subsolution part. In fact, in this case, we have 
\begin{equation}\label{char-eq2}
    \sup_{y\in B_\delta(x)\setminus \{x\}}\dfrac{(u(x)-u(y))_+}{d(x, y)}\geq \sup_{y\in \partial B_r(x)}\dfrac{(u(x)-u(y))_+}{r}
\end{equation}
for all $\delta>r>0$ such that $\ol{B_{r}(x)}\subset\Omega$. Letting $\delta\to 0$ in \eqref{char-eq2} and using $S^-_r u(x)\geq1$, we obtain $|\nabla^- u|(x)\geq 1$.  
    
The proof for the reverse implication $\Rightarrow$ is different from the subsolution case and uses the additional assumption of the lower boundedness of $u$ on balls. Fix arbitrarily $x_0\in\Omega$ and $r>0$ such that $\ol{B_r(x_0)}\subset \Omega$. Since $u$ is assumed to be bounded from below on closed balls, there exists $C>0$ such that
\[
\inf_{{B_r(x_0)}}u\geq u(x_0)-Cr.
\]
We now show that $S_r^-u(x_0)\geq 1$. Assume by contradiction that $q:=S_r^-u(x_0)<1$. We choose $c\in (q, 1)$ and define
\[
F(x):=u(x)+cd(x,x_0),\qquad x\in \ol{B_r(x_0)}.
\]
The function $F$ is continuous and bounded below on the complete metric space $\ol{B_r(x_0)}$. Since $q<1$, we have, for every $y\in\partial B_r(x_0)$,
\begin{equation}\label{char_eq4}
    F(y)=u(y)+cr\geq u(x_0)+(c-q)r. 
\end{equation}
Set $\alpha=(1-c)/2$. Let us choose $z\in\ol{B_r(x_0)}$ and $\vep>0$ small enough such that
\[ 
\alpha\vep<\frac{(c-q)r}{2}\quad \text{and}\quad  F(z)\leq \inf_{\ol{B_r(x_0)}}F+\alpha\vep. 
\]
Applying Ekeland's variational principle (\cite[Theorem 1.1]{E1}, \cite[Theorem 1]{E2}) to $F$, there exist $\hat{z}\in\ol{B_r(x_0)}$ such that $F(\hat{z})\leq F(z)$ and
\begin{equation}\label{char_eq5}
    F(\hat{z})\leq F(y)+\alpha d(\hat{z},y)\qquad \text{for all}\ y\in \ol{B_r(x_0)}. 
\end{equation}
Moreover,
\[ 
F(\hat{z})\leq F(z)\leq F(x_0)+\alpha\vep\leq u(x_0)+\frac{(c-q)r}{2}. 
\]
Hence, in view of \eqref{char_eq4}, we have $\hat{z}\notin \partial B_r(x_0)$. Furthermore, \eqref{char_eq5} implies
\[ 
\alpha\geq \limsup_{y\to\hat{z}}\dfrac{(F(\hat{z})-F(y))_+}{d(\hat{z},y)}\geq \limsup_{y\to\hat{z}}\dfrac{(u(\hat{z})-u(y))_+}{d(\hat{z},y)}-c\geq 1-c=2\alpha, 
\]
which is a contradiction. Thus, $S^-_ru(x_0)\geq 1$ for all $r>0$ such that $\ol{B_r(x_0)}\subset \Omega$.

The assertion (iii) is a direct consequence of (i) and (ii). The subsolution characterization is already established in (i). For the supersolution part, note that, as a Monge solution, $u$ satisfies \eqref{char-eq4} on $\ol{B_r(x)}$, which implies its boundedness from below on $\ol{B_r(x)}$. Then one can apply (ii) to conclude the equivalence for the supersolution property. 
\end{proof}

\begin{rmk}
The sphere characterization of Monge solutions (resp., Monge subsolution, Monge supersolution) in Theorem~\ref{thm:char} also holds for the eikonal equation
\begin{equation}\label{eikonal_c}
   |\nabla^-u|(x)=c,\quad x\in\Omega, 
\end{equation}
for each $c>0$. This is due to a straightforward fact that $|\nabla^-(u/c)|=|\nabla^-u|/c$.
\end{rmk}

\begin{rmk}
Under the assumptions of Theorem~\ref{thm:char}, one can also establish an analogous characterization by replacing the sphere-based subslope $S_r^-u$ with the punctured-ball subslope
\[
D_r^-u(x):=\sup_{y\in B_r(x)\setminus\{x\}}\frac{(u(x)-u(y))_+}{d(x,y)},
\quad x\in\Omega.
\]
Indeed, for any fixed $x\in\Omega$ and $r>0$ satisfying $\ol{B_r(x)}\subset\Omega$,
\[
D_r^-u(x)\le1
\quad\Longleftrightarrow\quad
S_\rho^-u(x)\le1
\ \text{for every }\rho\in(0,r].
\]
Consequently, $u$ is a Monge subsolution of \eqref{eikonal1} if and only if
$D_r^-u(x)\le 1$ for all such $x$ and $r$. An analogous characterization also holds for Monge supersolutions. We remark that the nonlocal slope $D_r^-u$ was introduced in \cite{LTG} in the study of global eikonal equations on metric spaces.
\end{rmk}


\section{Stability of Monge solutions}\label{sec:stability}

\subsection{Eikonal equations with constant speed}\label{subsec:eikonal1}
In this section, we give a proof of our main result, Theorem~\ref{thm:dom-stability}. Let us begin with proving the stability of the Monge subsolution property, for which we can relax the assumptions on $(\X_0, d_0)$. 

\begin{prop}[Stability for Monge subsolutions]\label{prop:sub}
    Suppose that $(\X, d)$ is a metric space. For each $\vep\geq 0$, let $\X_\vep\subset \X$ be a complete length space with metric $d_\vep$ and $\Omega_\vep\subsetneq\X_\vep$ be a domain in $(\X_\vep, d_\vep)$. Assume that (A1)(A2) in Theorem~\ref{thm:dom-stability} hold. Let $u_\vep\in \lipl(\Omega_\vep)$ be a Monge subsolution of $|\nabla_{d_\vep} u_\vep|=1$ in $\Omega_\vep$ for each $\vep>0$. Assume that $u_\vep\to u$ as $\vep\to 0$ in the sense of \eqref{unif-converge}. Then $u$ is a Monge subsolution of $|\nabla_{d_0} u|=1$ in $\Omega_0$. 
\end{prop}

\begin{proof}
Fix $x_0\in \Omega_0$ and $r_0>0$ arbitrarily such that 
\begin{equation}\label{dom-stability eq1}
    V_0=\{y\in \X_0: d_0(x_0, y)\leq r_0\}\subset \Omega_0.
\end{equation}
By (A2), there exist $x_\vep\in \Omega_\vep$ and $r_\vep>0$ such that \eqref{dom-ball} holds and $r_\vep\to r_0$, $d(x_\vep, x_0)\to 0$ and $d_H(V_\vep, V_0)\to 0$ as $\vep\to 0$. 
Then, for any fixed $y_0\in V_0$ satisfying $d_0(x_0, y_0)=r_0$, one can take $y_\vep\in V_\vep\subset \Omega_\vep$ such that $d(y_\vep, y_0)\to 0$ as $\vep\to 0$. 
It follows from (A1) that $d_\vep(x_\vep, y_\vep)\to d_0(x_0, y_0)$ as $\vep\to 0$. By Theorem \ref{thm:char} for Monge subsolutions, we obtain
\[
u_\vep(x_\vep)-u_\vep(y_\vep)\leq (u_\vep(x_\vep)-u_\vep(y_\vep))_+\leq d_\vep(x_\vep, y_\vep). 
\]
Letting $\vep\to 0$, we get
\begin{equation}\label{dom-stability eq2}
    u(x_0)-u(y_0)\leq (u(x_0)-u(y_0))_+\leq d_0(x_0, y_0).
\end{equation}
We can use the same argument to show that, for any $r>0$ small such that $\{z\in \X_0: d_0(x_0, z)\leq 3r\}\subset \Omega_0$ holds, 
\begin{equation}\label{loc-lip0}
    |u(x)-u(y)|\leq d_0(x, y)\quad \text{for all $x, y\in \Omega_0$ with $d_0(x_0, x), d_0(x_0, y)< r$.}
\end{equation}
In fact, these $x, y$ satisfy $d_0(x, y)<2r$ and 
\[
\{z\in \X_0: d_0(z, x)\leq 2r\}\subset \{z\in \X_0: d_0(x_0, z)\leq 3r\}\subset \Omega_0,
\]
and therefore the argument above yields $u(x)-u(y)\leq d_0(x, y)$. Since 
\[
\{z\in \X_0: d_0(z, y)\leq 2r\}\subset \{z\in \X_0: d_0(x_0, z)\leq 3r\}\subset \Omega_0,
\]
exchanging the roles of $x, y$, we deduce \eqref{loc-lip0}. This shows that $u\in \lipl(\Omega_0)$ with respect to the metric $d_0$. It also follows from \eqref{dom-stability eq2} that
\[
\sup_{d_0(x_0, y)=r_0}\dfrac{(u(x_0)-u(y))_+}{r_0}\leq 1.    
\]
This amounts to saying that $S_{r_0}^- u(x_0)\leq 1$, where $S^-_{r_0}u$ is given by \eqref{1-slope2} with respect to $(\X_0, d_0)$. By Theorem~\ref{thm:char} again, we see that $u$ is a Monge subsolution of $|\nabla u|=1$ in $\Omega_0$. 
\end{proof}

We now prove Theorem~\ref{thm:dom-stability}. 
\begin{proof}[Proof of Theorem~\ref{thm:dom-stability}]
In view of Proposition \ref{prop:sub}, it suffices to show the supersolution property of $u$. Fix $x_0\in \Omega_0$ and $r_0>0$ arbitrarily such that \eqref{dom-stability eq1} holds. We again use (A2) to get $x_\vep\in \Omega_\vep$ and $r_\vep>0$ such that $V_\vep\subset \Omega_\vep$, and $d(x_\vep, x_0)\to 0$, $r_\vep\to r_0$ and $d_H(V_\vep, V_0)\to 0$ as $\vep\to 0$. 

By Theorem~\ref{thm:char}, for each $\sigma>0$ we can find  $\tilde{y}_\vep\in V_\vep$ with $d_\vep(x_\vep, \tilde{y}_\vep)=r_\vep$ for $\vep>0$ small such that 
\begin{equation}\label{dom-convergence3}
    u_\vep(x_\vep)-u_\vep(\tilde{y}_\vep)\geq (1-\sigma) d_\vep(x_\vep, \tilde{y}_\vep).
\end{equation}
Since $d_H(V_\vep, V_0)\to 0$, we can find a sequence $\{\tilde{y}_\vep'\}$ in $V_0$ such that $d(\tilde{y}_\vep,\tilde{y}_\vep')\to 0$ as $\vep\to 0$. By the compactness of $V_0$, passing to a subsequence, we have $d_0(\tilde{y}'_\vep, \tilde{y}_0)\to 0$ for some $\tilde{y}_0\in V_0$ as $\vep\to 0$. By \eqref{compat-metric}, this yields $d(\tilde{y}'_\vep, \tilde{y}_0)\to 0$ and $d(\tilde{y}_\vep, \tilde{y}_0)\to 0$ as $\vep\to 0$ again up to a subsequence. By (A1), we further get $d_\vep(x_\vep, \tilde{y}_\vep) \to d_0(x_0, \tilde{y}_0)$ as $\vep\to 0$. Hence, $d_0(x_0,\tilde{y}_0)=r_0$.  Applying the convergence $u_\vep\to u$ to \eqref{dom-convergence3}, we obtain
\[
u(x_0)-u(\tilde{y}_0)\geq (1-\sigma) d_0(x_0, \tilde{y}_0).
\]
Passing to the limit as $\sigma\to 0$, we deduce
\[
\sup_{d_0(x_0, y)=r_0}\dfrac{(u(x_0)-u(y))_+}{r_0}\geq 1.
\]
This directly yields $S^-_{r_0} u(x_0)\geq 1$ for all $x_0\in \Omega_0$ and $r_0>0$ small. Using Theorem~\ref{thm:char} once again, we conclude that $u$ satisfies $|\nabla^- u|(x_0)\geq 1$ in the metric $d_0$. 
\end{proof}

Compared to Proposition~\ref{prop:sub}, the stability of supersolutions established in the proof of Theorem~\ref{thm:dom-stability} requires an additional properness assumption on $(\X_0, d_0)$. Although no properness assumption is imposed on the approximating spaces $(\X_\vep, d_\vep)$ in Theorem~\ref{thm:dom-stability}, properness of the limit space $(\X_0, d_0)$ is needed. Without this assumption, the stability result may fail, as shown in Example~\ref{ex:L2_instability}.

\subsection{More general equations}
Our results in the preceding subsection can be extended to more general eikonal equations and Hamilton-Jacobi equations in metric spaces. We assume that the ambient space $(\X, d)$ is a complete length space. Let us first consider \eqref{eikonal} with a domain $\Omega\subsetneq \X$. Let $f: \X\to \R$ be a bounded function satisfying
\begin{equation}\label{f-lower}
    \alpha:=\inf_{x\in \X} f(x)>0. 
\end{equation}
Note that it is sufficient to assume $f$ bounded with $\alpha=\inf_{\Oba} f>0$ to consider Monge solutions of \eqref{eikonal}. Here, we extend this condition to the entire space $\X$ for our later convenience of asymptotic analysis, especially in studying stability with respect to variations of the domain. 

The general equation \eqref{eikonal} can be handled by introducing the optical length function $L_f$ in \eqref{opt-length}. It is not difficult to show that $L_f$ is a metric in $\X$ and $(\X, L_f)$ is a length space; see \cite[Lemma 2.1, Remark 2.2]{LShZ2} for details. By adopting a metric change from $d$ to $L_f$, we can convert \eqref{eikonal} to \eqref{eikonal1}, which largely facilitates our generalization of the stability result in Section~\ref{subsec:eikonal1}. 

Since $f$ is bounded in $\X$, there exists $M>0$ such that
\begin{equation}\label{f-upper}
    M=\sup_{x\in \X} f(x). 
\end{equation}
Combining \eqref{f-lower} and \eqref{f-upper}, we also see that
\[
\alpha d(x, y)\leq L_f(x, y)\leq Md(x, y) \quad\text{for $x, y\in \X$.}
\]
In other words, $L_f$ and $d$ are bi-Lipschitz equivalent. Hence, changing metric $d$ to $L_f$ preserves the original topology of the space with respect to $d$. 

Let us now proceed to introduce more details about our stability result for \eqref{eikonal} with respect to both $f$ and $\Omega$. For each $\vep> 0$, let $(\X_\vep,d_\vep)\subset (\X,d)$ be a complete length space and $\Omega_\vep\subsetneq\X_\vep$ be a domain in $(\X_\vep, d_\vep)$. Let $f_\vep: \X_\vep \to \R$ be a bounded positive function for each $\vep> 0$. Assume that 
\begin{equation}\label{f-vep-bound}
    \alpha=\inf_{x\in \X_\vep, \vep> 0} f_\vep(x)\leq \sup_{x\in \X_\vep , \vep> 0} f_\vep(x)=M \quad\text{for some $\alpha, M>0$.}
\end{equation}
For each $\vep> 0$, the function $f_\vep$ defines an optical length function $L_{f_\vep}$ as given by \eqref{opt-length} with metric $d_\vep$. More precisely, for $x, y\in \X_\vep$ we define
\begin{equation}\label{optical-length-eps}
    L_{f_\vep}(x, y)=\inf\left\{\int_{0}^{\ell} f_\vep(\gamma(s))\, ds: \gamma\in \Gamma_\vep(x, y), \ \gamma(0)=x, \ \gamma(\ell)=y\right\},
\end{equation}
where $\Gamma_\vep(x, y)$ represents the class of arc-length parametrized curves joining $x, y$ with respect to the metric $d_\vep$, similar to $\Gamma(x, y)$ given in \eqref{curve-set}. 

We can therefore obtain a stability result for Monge solutions to \eqref{eikonal} by replacing $d_\vep$ by $L_{f_\vep}$ in Theorem~\ref{thm:dom-stability}, as stated below. 

\begin{thm}[Stability of Monge solutions to general eikonal equations]\label{thm:dom-stability-gen}
    Let $(\X,d)$ be a complete length space, and let $(\X_\vep,d_\vep)\subset(\X,d)$ be a complete length space for $\vep>0$. Let $f_\vep\in C(\X_\vep)$ satisfy \eqref{f-vep-bound} and $L_{f_\vep}$ be defined by \eqref{optical-length-eps}. Let $(\X_0,\bar d)\subset(\X,d)$ be a proper geodesic space, and assume that 
    \[ 
    \bar d(x_j,x)\to 0 \quad \text{implies} \quad d(x_j,x)\to0 \quad\text{for any $x_j, x\in\X_0$}. 
    \] 
    Let $\Omega_\vep\subsetneq\X_\vep$ be domains in $(\X_\vep,d_\vep)$, $\vep>0$ and $\Omega_0\subsetneq\X_0$ be a domain in $(\X_0,\bar d)$, satisfying the following conditions.
    \begin{enumerate}
    \item[(A1')] If $x_\vep, y_\vep\in \Omega_\vep$, $x_0, y_0\in \Omega_0$ satisfy \eqref{dom-convergence1}, then
    \begin{equation}\label{dom-convergence2-gen}
        L_{f_\vep}(x_\vep, y_\vep)\to \bar{d}(x_0, y_0) \quad \text{as $\vep\to 0$.}
    \end{equation}
    \item[(A2')] For any $x_0\in \Omega_0$ and $r_0>0$ small satisfying $V_0':=\{x\in \X_0: \bar{d}(x, x_0)\leq r_0\}\subset \Omega_0$, there exist $x_\vep\in \Omega_\vep$, $r_\vep>0$ such that 
    \[
    V_\vep':=\{x\in \X_\vep: L_{f_\vep}(x, x_\vep)\leq r_\vep\}\subset \Omega_\vep,
    \]
    and $r_\vep\to r_0$, $d(x_\vep, x_0)\to 0$ and $d_H(V_\vep', V_0')\to 0$ as $\vep\to 0$.
    \end{enumerate}
    Let $u_\vep\in \lipl(\Omega_\vep)$ be a Monge solution of $|\nabla_{d_\vep} u_\vep|=f_\vep$ in $\Omega_\vep$ for each $\vep>0$. Assume that $u_\vep\to u$ as $\vep\to 0$ in the sense of \eqref{unif-converge}. Then $u$ is a Monge solution of $|\nabla_{\bar{d}} u|=1$ in $\Omega_0$. If in addition there exists a positive bounded function $f_0\in C(\Omega_0)$ and a metric $d_0$ in $\X_0$ such that $d_0$ and $\bar{d}$ are locally bi-Lipschitz equivalent and for any fixed $x\in \Omega_0$,
    \begin{equation}\label{density}
        \lim_{d_0(x, y)\to 0}\frac{\bar{d}(x, y)}{d_0(x, y)}= f_0(x),     
    \end{equation}
    then $u$ is a Monge solution of $|\nabla_{d_0} u|=f_0$ in $\Omega_0$ with respect to $d_0$.
\end{thm}

\begin{proof}
Since $f_\vep$ is continuous in $(\X_\vep, d_\vep)$ for all $\vep> 0$, one can use Definition~\ref{defi:monge} to show that $u_\vep$ is a Monge solution of $|\nabla u_\vep|=f_\vep$ in $\Omega_\vep$  with respect to $d_\vep$ if and only if it is a Monge solution of $|\nabla u_\vep|=1$ in $\Omega_\vep$ with respect to $L_{f_\vep}$. Moreover, under the geodesic structure of $(\X_0, \bar{d})$, conditions (A1') and (A2') match conditions (A1) and (A2) in Theorem~\ref{thm:dom-stability} with $d_\vep=L_{f_\vep}$ for $\vep>0$ and $d_0=\bar{d}$. Hence, this stability result can be obtained by directly applying Theorem~\ref{thm:dom-stability}. 

The last statement regarding the Monge solution property of $u$ with respect to $d_0$ is also a direct consequence. In fact, by using \eqref{density} and the definition of Monge solutions, we have, for each $x\in \Omega_0$,
\[
|\nabla_{d_0}^- u|(x)=\limsup_{d_0(x, y)\to 0} \frac{(u(x)-u(y))_+}{d_0(x, y)}=f_0(x)\limsup_{\bar{d}(x, y)\to 0} \frac{(u(x)-u(y))_+}{\bar{d}(x, y)}=f_0(x)|\nabla^-_{\bar{d}} u|=f_0(x).
\]
This completes the proof. 
\end{proof}

\begin{rmk}
Suppose that $(\X,d)$ is a proper geodesic space and $\Omega\subsetneq \X$ is a domain such that $\X_\vep=\X$ and $\Omega_\vep=\Omega$ for all $\vep\geq 0$. Then Theorem~\ref{thm:dom-stability-gen} reduces to the stability of solutions to \eqref{eikonal} under perturbations of the inhomogeneous term $f$ alone. Let $L_{f_\vep}$ be defined as in \eqref{optical-length-eps} with $f_\vep\in C(\X)$ uniformly bounded for all $\vep\geq 0$. In this setting, an assumption stronger than (A1') and (A2') is as follows:
\begin{equation*}
    L_{f_\vep}(\cdot, K)\to L_{f_0}(\cdot, K) \quad \text{uniformly in $\Omega$ as $\vep\to 0$, for every compact set $K\subset \Omega$.}
\end{equation*}
In particular, if $f_\vep$ converges to a bounded continuous function $f_0$ uniformly in $\Omega$, then the uniform limit $u_0$ of a family of Monge solutions $u_\vep$ to $|\nabla u|=f_\vep$ in $\Omega$ is a Monge solution corresponding to $f_0$. This is consistent with the classical stability theory for viscosity solutions of eikonal equations in Euclidean space. 
\end{rmk}

\begin{rmk}
We can drop the continuity assumption on $f_\vep$ in Theorem~\ref{thm:dom-stability-gen}. Under the condition \eqref{f-vep-bound}, one can define Monge solutions of $|\nabla u_\vep|=f_\vep$ in $d_\vep$ to be Monge solutions of $|\nabla u_\vep|=1$ in $L_{f_\vep}$; see \cite[Definition 3.1]{LShZ2}. Under this definition, the corresponding stability result for discontinuous eikonal equations is essentially the same as Theorem~\ref{thm:dom-stability}. 
        
It is possible to further relax the boundedness of $f_\vep$ if the metric $L_{f_\vep}$ can be properly defined in $\X$ via path integrals without altering the space topology. In this case, \eqref{dom-convergence2} should be replaced with the stronger condition \eqref{dom-convergence2-gen}. We do not pursue this direction here, and instead refer to \cite{LShZ2} again for related discussions about Monge solutions to general discontinuous or singular eikonal equations in metric spaces. 
\end{rmk}

Finally, let us briefly discuss a further generalization of Theorem~\ref{thm:dom-stability-gen} for general Hamilton-Jacobi equations in the form of \eqref{general eq},
where $H: \X\times \R\times [0,\infty)\to \R$ is a continuous function satisfying the assumption below.  
\begin{enumerate}
    \item[(H1)] For each $x\in \X$ and $r\in \R$, $p\mapsto H(x, r, p)$ is strictly increasing in $[0, \infty)$. 
\end{enumerate}

Under (H1), let us consider Monge solution (resp., Monge subsolution, Monge supersolution) $u$ of \eqref{general eq} in $\Omega$, defined to be $u\in \lipl(\Omega)$ such that 
\begin{equation}\label{Monge general}
    H(x, u(x), |\nabla^- u|(x))=0\quad (\text{resp.}\; \leq 0,\; \geq 0) , \quad x\in \Omega. 
\end{equation}
Assume that $|\nabla^- u|(x_0)>0$ at some $x_0\in \Omega$. Then, thanks to the strict monotonicity (H1), we can deduce, at least formally,  
\begin{equation}\label{general loc eq}
    |\nabla u|(x)=h(x),\quad \text{$x\in B_r(x_0)$,}
\end{equation}
where $r>0$ is small and $h: B_{r}(x_0)\to \R$ is a continuous function. In fact, $u$ is a Monge solution of \eqref{general eq} restricted to $B_r(x_0)$ if and only if $u$ is a Monge solution of \eqref{general loc eq}. By the continuity of $h$, we can choose $r>0$ further small so that 
\[
\inf_{x\in B_r(x_0)} h(x)>0. 
\]
With this reformulation, we can generalize our stability results in Theorem~\ref{thm:dom-stability-gen} for Monge solutions $u$ of \eqref{general eq} provided that $|\nabla^- u|>0$ in $\Omega$. We omit further details here.

\subsection{Example: a two-dimensional square lattice}\label{sec:example}

We illustrate Theorem~\ref{thm:dom-stability} and Theorem~\ref{thm:dom-stability-gen} through a family of rescaled square lattices. This setting is also closely related to homogenization on periodic lattices as the mesh size tends to zero. 

Let $(\X,d)$ be the Euclidean space $\R^2$ equipped with $1$-norm $\|\cdot\|_1$, i.e., 
\[
d(x, y)=\|x-y\|_1=|x_1-y_1|+|x_2-y_2|, \quad x, y\in \R^2.
\]
Rescaling the square lattice $\L^2:=(\Z\times\R)\cup(\R\times\Z)$ with  $\vep\in(0,1)$, we take
\begin{equation}\label{X-eps}
    \X_\vep:=\vep \L^2=(\vep\Z\times\R)\cup(\R\times\vep\Z),
\end{equation}
equipped with the intrinsic metric $d_\vep$ induced by $d$. Then, $(\X_\vep,d_\vep)$ is a proper geodesic metric space. Moreover, we have
\[
\|x-y\|_1\le d_\vep(x,y)\le \|x-y\|_1+\vep, \quad x,y\in\X_\vep.
\]
Let $f\in C(\R^2)$ be a positive $\Z^2$-periodic function. For $\vep>0$, define
\begin{equation}\label{periodic_f_vep}
    f_\vep(x):=f\!\left(\frac{x}{\vep}\right),\quad x\in\X_\vep,
\end{equation}
and denote by $L_{f_\vep}$ the optical length function associated with $f_\vep$ and $d_\vep$ as in \eqref{optical-length-eps}. Set
\begin{equation}\label{effective-coefficients}
    a_1:=\int_0^1f(t,0)\,dt,\qquad
    a_2:=\int_0^1f(0,t)\,dt.
\end{equation}
These constants represent the optical costs of traversing one unit cell in the horizontal and vertical directions, respectively. By periodicity, for any two vertices $x, y\in \vep\Z^2$ given by $x=(\vep h_x, \vep k_x)$, $y=(\vep h_y, \vep k_y)$ with $(h_x, k_x), (h_y, k_y)\in \Z^2$, we easily see that
\begin{equation}\label{L_vep_metric}
    L_{f_\vep}(x, y)=\vep a_1|h_x-h_y|+\vep a_2|k_x-k_y|.
\end{equation}
This suggests that the limiting metric $\bar{d}$ of $L_{f_\vep}$ on $\R^2$ as $\vep\to 0$ should be
\begin{equation}\label{limit_metric}
    \bar d(x,y):=a_1|x_1-y_1|+a_2|x_2-y_2|,\quad \text{for } x=(x_1, x_2),\ y=(y_1, y_2)\in\R^2.
\end{equation}
In fact, letting
\[
a^\ast=\max\{a_1, a_2\}, \quad M:=\max_{\R^2} f, \quad C_1=a^\ast+M,
\]
we have 
\begin{equation}\label{estimate:L_vep-d_bar}
    \left|L_{f_\vep}(x,y)-\bar{d}(x,y)\right|\leq C_1\vep
    \quad\text{for all $x,y\in\X_\vep$ and $\vep>0$ small.}
\end{equation}
Indeed, given $x, y\in \X_\vep$, we can choose vertices
$\hat x,\hat y\in\vep\Z^2$
lying on the same grid edges as $x$ and $y$, respectively, such that
$d_\vep(x,\hat x), d_\vep(y,\hat y)\le \vep/2$, which yields
\[
|L_{f_\vep}(x,y)-L_{f_\vep}(\hat x,\hat y)|\le M\vep.
\]
Since $|\bar d(x,y)-\bar d(\hat x,\hat y)|\le a^*\vep$ holds and \eqref{L_vep_metric} implies $L_{f_\vep}(\hat x,\hat y)=\bar d(\hat x,\hat y)$, we obtain \eqref{estimate:L_vep-d_bar} immediately.

Noticing that $f\in C(\R^2)$ is positive and periodic, we further obtain 
\begin{equation}\label{bi-Lipschitz_d_1-d_bar}
    a_*d(x, y)\le \bar{d}(x,y)\le a^* d(x, y), \quad x,y\in\R^2,
\end{equation}
where we take $a_*:=\min\{a_1, a_2\}>0$. In other words, $\bar d$ is bi-Lipschitz equivalent to the metric induced by $\|\cdot\|_1$, and therefore $(\R^2, \bar d)$ is a proper geodesic space.

Let $\Omega_0\subset\R^2$ be a domain and $\Omega_\vep:=\Omega_0\cap\X_\vep$.
We assume that $\Omega_\vep$ is a domain in $\X_\vep$ when $\vep>0$ is sufficiently small. Next we verify assumptions (A1') and (A2') in Theorem~\ref{thm:dom-stability-gen}.

To verify (A1'), suppose that there are $x_0, y_0\in \Omega_0$ and $x_\vep, y_\vep\in \Omega_\vep\subset \Omega_0$ such that \eqref{dom-convergence1} holds. In view of \eqref{estimate:L_vep-d_bar} and \eqref{bi-Lipschitz_d_1-d_bar}, we thus obtain 
\begin{align*}
|L_{f_\vep}(x_\vep, y_\vep)-\bar{d}(x_0, y_0)|\leq &\ |L_{f_\vep}(x_\vep, y_\vep)-\bar{d}(x_\vep, y_\vep)|+|\bar{d}(x_\vep, y_\vep)-\bar{d}(x_0, y_0)|\\
\leq&\ C_1\vep+a^\ast(d(x_\vep,x_0)+d(y_\vep,y_0))\to 0,    
\end{align*}
which yields \eqref{dom-convergence2-gen} in (A1') immediately. 

Let us now show that (A2') holds. Fix $x_0\in\Omega_0$ and $r_0>0$ such that
\[
V_0':=\{x\in\R^2:\bar{d}(x,x_0)\leq r_0\}\subset\Omega_0.
\]
For sufficiently small $\vep>0$, choose $x_\vep\in\Omega_\vep$ satisfying $d(x_\vep, x_0)\leq\vep/2$, and let
\[
r_\vep:=r_0-2C_1\vep>0,\quad
V_\vep':=\{x\in\X_\vep:L_{f_\vep}(x,x_\vep)\leq r_\vep\}.
\]
By \eqref{estimate:L_vep-d_bar} and \eqref{bi-Lipschitz_d_1-d_bar}, any $x\in V_\vep'$ satisfies 
\[
\bar{d}(x,x_0)\leq\bar{d}(x,x_\vep)+\bar{d}(x_\vep,x_0)\leq L_{f_\vep}(x,x_\vep)+C_1\vep+a^*d(x_\vep,x_0)\leq r_0-C_1\vep+\frac{a^*}{2}\vep
\leq r_0, 
\]
which amounts to saying that $V_\vep'\subset V_0'\cap\X_\vep\subset\Omega_\vep$. 

It remains to prove $d_H(V_\vep', V_0')\to 0$ as $\vep\to 0$. Set $\eta_\vep:=(3C_1+a^\ast)\vep$ so that  $\eta_\vep<r_0$ when $\vep>0$ is taken small. Fix $z\in V_0'\setminus \{x_0\}$ arbitrarily and take
\[
z':=x_0+\left(1-\frac{\eta_\vep}{\bar{d}(z,x_0)}\right)_+(z-x_0).
\]
We can easily obtain the estimates that
\begin{equation}\label{appox:z}
    \bar{d}(z,z')\leq\eta_\vep,\quad
    \bar{d}(z',x_0)\leq r_0-\eta_\vep.
\end{equation}
Choose $z_\vep\in\X_\vep$ such that $d(z_\vep,z')\leq\vep/2$. By \eqref{estimate:L_vep-d_bar}, \eqref{bi-Lipschitz_d_1-d_bar}, \eqref{appox:z} as well as the choice of $r_\vep$, we then have
\begin{align*}
L_{f_\vep}(z_\vep,x_\vep)
&\leq \bar{d}(z_\vep,x_\vep)+C_1\vep\leq \bar{d}(z_\vep,z')+\bar{d}(z',x_0)
 +\bar{d}(x_0,x_\vep)+C_1\vep\\
&\leq \frac{1}{2}a^*\vep+(r_0-\eta_\vep)
 +\frac{1}{2}a^*\vep+C_1\vep=r_0-2C_1\vep=r_\vep.
\end{align*}
This shows that $z_\vep\in V_\vep'$. Furthermore, the lower bound in \eqref{bi-Lipschitz_d_1-d_bar} implies
\[
d(z,z_\vep)\leq d(z,z')+d(z',z_\vep)
\leq \frac{1}{a_*}\eta_\vep+\frac{\vep}{2}.
\]
Since we have shown that $V_\vep'\subset V_0'$, it then follows from the arbitrariness of $z\in V_0'\setminus \{x_0\}$ that 
\[
d_H(V_\vep',V_0') \leq \frac{1}{a_*}\eta_\vep+\frac{\vep}{2} \to 0, \quad \text{as $\vep\to 0$.}
\]
Assumptions (A1') and (A2') have now been verified. Applying Theorem~\ref{thm:dom-stability-gen} to this setting, we obtain the following result. 

\begin{prop}\label{prop:lattice-example}
    Let $(\X,d)$ be $\R^2$ with $d$ induced by the $1$-norm $\|\cdot\|_1$. For $\vep>0$ small, let $\X_\vep$ be given by \eqref{X-eps} equipped with the intrinsic metric $d_\vep$ and $f_\vep$ be given by \eqref{periodic_f_vep} for a positive continuous $\Z^2$-periodic function $f$ in $\R^2$.  Let $\bar{d}$ be defined as in \eqref{limit_metric}. Let $\Omega_0\subset\R^2$ be a domain and $\Omega_\vep:=\Omega_0\cap\X_\vep$. Assume $\Omega_\vep$ is a domain for all $\vep>0$ sufficiently small. If $u_\vep$ is a Monge solution of
    \begin{equation*}
        |\nabla_{d_\vep} u_\vep|=f_\vep
        \quad\text{in $\Omega_\vep$},
    \end{equation*}
    and $u_\vep\to u_0$ as $\vep\to 0$ in the sense of \eqref{unif-converge}, then $u_0$ is a Monge solution of    \begin{equation}\label{eq:limit_eikonal}
        |\nabla_{\bar{d}} u_0|=1
        \quad\text{in $\Omega_0$}.
    \end{equation}
\end{prop}

If $a_1=a_2=a$ for some $a>0$, then the Monge solution $u_0$ of \eqref{eq:limit_eikonal} is also a Monge solution of $|\nabla_d u_0|=a$ in $\Omega_0$. In particular, this example illustrates the stability result in Theorem~\ref{thm:dom-stability} in the special case where $f\equiv 1$ in $\R^2$ and $a_1=a_2=1$.

On the other hand, if $a_1\neq a_2$, then the limiting Hamiltonian in \eqref{eq:limit_eikonal} is anisotropic because the metric $\bar d$ assigns different weights to the horizontal and vertical directions. In this case, the equation cannot be expressed in the form $|\nabla_d u|=f_0$, since there is no $f_0\in C(\Omega_0)$ satisfying \eqref{density}. In fact, \eqref{eq:limit_eikonal} can be written explicitly in Euclidean coordinates as
\[
\max\left\{\frac{1}{a_1}\left|\frac{\partial u}{\partial x_1}(x)\right|,\ \frac{1}{a_2}\left|\frac{\partial u}{\partial x_2}(x)\right|\right\}=1, \quad x=(x_1, x_2)\in \Omega_0.
\]
As seen in \eqref{effective-coefficients}, the homogenized PDE is determined only by the values of $f(\cdot,0)$ and $f(0,\cdot)$, rather than by the full values of $f$. This feature arises from the discrete structure of the approximating spaces and differs from the standard Euclidean case.

This example provides a simple illustration of how to extend the homogenization theory for nonlinear PDEs in Euclidean spaces to general geodesic spaces equipped with an appropriate periodic structure. We refer to the monograph \cite{TrBook} for classical homogenization results for Hamilton-Jacobi equations. Although our setting involves the additional perturbation of the domain $\Omega_0$ by the lattice approximations $\Omega_\vep$, the estimate \eqref{estimate:L_vep-d_bar}, together with the mesh size $\vep$, suggests that the homogenization error is of order $O(\vep)$ for the associated Dirichlet problems. This is consistent with the connection between homogenization and stable norms in metric geometry pointed out in a recent work by Tran and Yu \cite{TrYu}. Through personal communications with Nakayasu, we also learned about his partial results on the homogenization of Hamilton-Jacobi equations on fractals using a similar space-approximation approach. Related results have been announced at several seminars, including, for instance, \cite{Na}.


\section{Perron's method for Monge solutions}\label{sec:perron}

\subsection{Eikonal equations}

The characterization result in Theorem~\ref{thm:char} enables us to further establish Perron's method for Monge solutions to eikonal equations \eqref{eikonal}, as stated in Theorem~\ref{perron:gen_eikonal}.  

\begin{prop}[Pointwise extrema of subsolutions]\label{prop:gen-eikonal1}
    Suppose that $(\X, d)$ is a complete length space and $\Omega\subsetneq\X$ is a domain. Assume that $f\in C(\Omega)$ is nonnegative. Let $\mathcal{U}$ be a nonempty collection of Monge subsolutions of \eqref{eikonal}. Suppose that 
    \begin{equation}\label{point-sup}
        w_1(x):=\sup\{u(x): u\in\mathcal{U}\}<\infty \quad \text{for each $x\in \Omega$,}
    \end{equation}
    \begin{equation}\label{point-inf}
        w_2(x):=\inf\{u(x): u\in\mathcal{U}\}>-\infty \quad \text{for each $x\in \Omega$.}
    \end{equation}
    Then, $w_1$ and $w_2$ are Monge subsolutions of \eqref{eikonal}. 
\end{prop}

\begin{proof} We divide our proof into two steps. The first step is for the case $f\equiv 1$, and the second is to discuss the case of a general $f$.  

\ul{\emph{Step 1.}} We prove all the statements for \eqref{eikonal1}. In this case, the proof is similar to that of Proposition~\ref{prop:sub}. Based on the proof of \eqref{loc-lip0}, one can show that $w_1,w_2\in \lipl(\Omega)$. Indeed, for any $u\in \U$ and any fixed $x_0\in\Omega$, $r>0$ such that $\ol{B_{r}(x_0)}\subset\Omega$, we apply Theorem~\ref{thm:char}(i) to obtain 
\begin{equation}\label{perron eq1}
    S^-_\rho u(x)\leq 1 \quad \text{for every $\rho\in (0, r]$} 
\end{equation}
such that $\ol{B_{\rho}(x)}\subseteq \ol{B_{r}(x_0)}$. In particular, we have $u(x)\leq u(y)+d(x, y)$ for all $x, y\in B_{r/4}(x_0)$ since $\ol{B_{r/2}(x)}\subset\ol{B_{r}(x_0)}\subset\Omega$. Similarly, we have $u(y)\leq u(x)+d(x, y)$. It follows that 
\[
|u(x)-u(y)|\leq d(x, y)\quad \text{ for all $x, y\in B_{r/4}(x_0)$.}
\]
Taking the supremum or infimum over all $u\in \U$, we obtain the same inequality for $w_1$ and $w_2$ in $B_{r/4}(x_0)$, which implies the local Lipschitz regularity of $w_1$ and $w_2$.  

Moreover, for every $u\in \mathcal{U}$, \eqref{perron eq1} also implies that
\begin{equation}\label{sup_subs}
    u(x_0)\leq u(y)+r\leq w_1(y)+r\quad \text{for all}\ y\in \partial B_r(x_0) 
\end{equation}
and
\begin{equation}\label{inf_subs}
    w_2(x_0)-r\leq u(x_0)-r\leq u(y)\quad \text{for all}\ y\in \partial B_r(x_0). 
\end{equation}
Taking the supremum over all such $u\in \mathcal{U}$ in \eqref{sup_subs} and the infimum over all such $u\in \mathcal{U}$ in \eqref{inf_subs} then yields
\[
w_1(x_0)\leq w_1(y)+r\qquad \text{and}\qquad w_2(x_0)\leq w_2(y)+r 
\]
which implies $S^-_rw_1(x_0)\leq 1$ and $S^-_rw_2(x_0)\leq 1$. Hence, we conclude that $w_1$ and $w_2$ are Monge subsolutions of \eqref{eikonal1}, thanks to Theorem~\ref{thm:char} again.

\ul{\emph{Step 2.}} We now study \eqref{eikonal} for general nonnegative $f\in C(\Omega)$. Fix $x_0\in\Omega$ arbitrarily and choose $r>0$ such that $f$ is bounded on $\ol{B_r(x_0)}\subset\Omega$. For any sufficiently small $\vep>0$, define
\[
f_\vep:=f+\vep
\quad\text{on}\ \ol{B_r(x_0)}.
\]
Then $\inf_{\ol{B_r(x_0)}} f_\vep\geq\vep>0.$ By the Tietze extension theorem, we may extend $f_\vep$ continuously to the whole space $\X$, still denoted by $f_\vep$, such that it is bounded and satisfies $\inf_{\X} f_\vep\geq\frac{\vep}{2}.$ We can therefore define the optical length metric $L_{f_\vep}$
as in \eqref{optical-length-eps}.

As in the proof of Theorem
\ref{thm:dom-stability-gen}, we can show that every $u\in\mathcal{U}$ is a Monge subsolution of \eqref{eikonal1} with respect to $L_{f_\vep}$ in $B_r(x_0)$. Hence, by the results obtained in Step 1, $w_1$ and $w_2$ are also Monge subsolutions of \eqref{eikonal1} with respect to $L_{f_\vep}$ in $B_r(x_0)$. Since $f_\vep$ is bounded and continuous,  $w_1$ and $w_2$ are locally Lipschitz in $B_r(x_0)$ with respect to the metric $d$, and
\[
|\nabla^-w_1|(x_0),|\nabla^-w_2|(x_0)
\leq f_\vep(x_0)=f(x_0)+\vep.
\]
Letting $\vep\to 0$, we obtain $|\nabla^-w_1|(x_0),|\nabla^-w_2|(x_0)\leq f(x_0)$. By the arbitrariness of $x_0\in\Omega$, we see that $w_1$ and $w_2$ are Monge subsolutions of \eqref{eikonal} in $\Omega$.
\end{proof}

\begin{prop}[Pointwise infimum of supersolutions]\label{prop:gen-eikonal2}
    Suppose that $(\X, d)$ is a complete length space and $\Omega\subsetneq\X$ is a domain. Assume that $f\in C(\Omega)$ is nonnegative. Let $\mathcal{U}$ be a nonempty collection of Monge supersolutions to \eqref{eikonal} such that 
    \begin{equation*}
        w_3(x):=\inf\{u(x): u\in\mathcal{U}\}>-\infty \quad \text{for each $x\in \Omega$.}
    \end{equation*}
    If $w_3\in\lipl(\Omega)$, then $w_3$ is a Monge supersolution of \eqref{eikonal}.
\end{prop}

\begin{proof}
Suppose, to the contrary, that $w_3$ is not a Monge supersolution of \eqref{eikonal}. Then there exists $x_0\in\Omega$ such that
\begin{equation}\label{cont_1_eik_gen}
    0\leq |\nabla^-w_3|(x_0)<c
\end{equation}
for some $c<f(x_0)$. By the continuity of $f$, we may choose $r>0$ arbitrarily small such that $\ol{B_r(x_0)}\subset\Omega$ and $f(x)\geq c$ for all $x\in B_r(x_0)$. Since every $u\in\mathcal{U}$ is a Monge supersolution of \eqref{eikonal}, it follows that
\[
|\nabla^-u|(x)\geq f(x)\geq c
\quad\text{for all }x\in B_r(x_0).
\]
In other words, each $u\in\mathcal{U}$ is a Monge supersolution of \eqref{eikonal_c} on $B_r(x_0)$. 

By Theorem~\ref{thm:char}(ii) applied to \eqref{eikonal_c} with a general constant $c>0$, we see that each $u\in \U$ satisfies 
\[
S_\rho^-u(x_0)=\sup_{y\in \partial B_\rho(x_0)} \frac{u(x_0)-u(y)}{\rho}\geq c.
\]
for $0<\rho<r$. Since $w_3\leq u$ in $\Omega$, it follows that 
\[
\sup_{y\in \partial B_\rho(x_0)} \frac{u(x_0)-w_3(y)}{\rho}\geq c.
\]
Taking the infimum of $u(x_0)$ over all $u\in \U$, we are led to $S_\rho^- w_3(x_0)\geq c$. It follows from Theorem~\ref{thm:char}(ii) (again for \eqref{eikonal_c}) that $|\nabla^-w_3|(x_0)\geq c$, which contradicts \eqref{cont_1_eik_gen}. Hence, $w_3$ is a Monge supersolution of \eqref{eikonal}.
\end{proof}

In Proposition~\ref{prop:gen-eikonal2}, we assume that $w_3$ is locally Lipschitz in $\Omega$. Without this assumption, the conclusion may fail, since the pointwise infimum of locally Lipschitz functions is not necessarily locally Lipschitz, as illustrated by the following example.

\begin{example}
    Let $(\X,d)=(\R^2,\|\cdot\|_2)$ and $\Omega=(0,1)\times (-1, 1)$. For $n\in\mathbb{N}$, we define
    \begin{equation*}
    u_n(x,y):=x+e^{-ny^2}. 
    \end{equation*}
    Note that $u_n\in C^\infty(\Omega)$ and one can calculate that
    \[ |\nabla^-u_n(x,y)|=|\nabla u_n(x,y)|=\sqrt{1+(2nye^{-ny^2})^2}\geq 1\qquad (x,y)\in\Omega. \]
    Thus, every $u_n$ is a Monge supersolution of \eqref{eikonal1} on $\Omega$. On the other hand, their pointwise infimum, given by 
    \[ 
    \inf_{n\in \mathbb{N}}u_n(x,y)=\begin{cases}
        x+1, & \text{if $y=0$,}\\
        x, & \text{if $y\neq 0$,}
    \end{cases} 
    \]
    is not a Monge supersolution of \eqref{eikonal1}, since it is not even continuous in $\Omega$. This example explains why the assumption $w_3\in\lipl(\Omega)$ is needed in Proposition~\ref{prop:gen-eikonal2} to ensure that $w_3$ is a Monge supersolution.
\end{example}

\begin{example}
    Another known example concerning the instability of metric viscosity supersolutions can be found in \cite[Example~2.1]{S} and \cite[Example~5.5]{GHN}, where $\X=\ell^2$, the space of square-summable sequences, and the stability of \eqref{eikonal1} is considered for
    \begin{equation}\label{example-dom}
    \Omega=\{x=(x_1, x_2, \ldots)\in \ell^2: x_i>0 \ \text{for all $i\geq 1$}\}.
    \end{equation}
    It was shown that 
    \[
    u_n(x)=x_n \quad \text{for $x=(x_1, x_2, \ldots)\in \Omega$}
    \]
    are metric viscosity solutions (equivalently, Monge solutions) to \eqref{eikonal1} for all $n\geq 1$, but the pointwise infimum $u\equiv 0$ is not a supersolution. 

    While this example may appear to contradict Proposition~\ref{prop:gen-eikonal2}, the key point is that $\Omega$ given in \eqref{example-dom} is not an open set in $\ell^2$. Note that, for any $x=(x_1, \ldots, x_{n-1}, x_n, x_{n+1},\ldots)\in\Omega$, we have $y_n:=(x_1, \ldots, x_{n-1}, 0, x_{n+1}, \ldots)\notin \Omega$ but $\|x-y_n\|_{\ell^2}=x_n$ is arbitrarily small when $n\geq 1$ is sufficiently large. It follows that $B_r(x)\cap \Omega^c\neq \emptyset$ for any $x\in \Omega$ and any $r>0$ small. In other words, $\Omega$ has empty interior, and therefore Proposition~\ref{prop:gen-eikonal2} does not apply to this example. 
\end{example}

\begin{thm}[Pointwise infimum of solutions]\label{thm:eikonal_gen_sol}
    Suppose that $(\X, d)$ is a complete length space and $\Omega\subsetneq\X$ is a domain. Assume that $f\in C(\Omega)$ is nonnegative. Let $\mathcal{U}$ be a nonempty collection of Monge solutions to \eqref{eikonal} such that 
    \[
    w(x):=\inf\{u(x): u\in\mathcal{U}\}>-\infty \quad \text{for each $x\in \Omega$.}
    \]
    Then, $w$ is a Monge solution of \eqref{eikonal}.
\end{thm}

\begin{proof}
By Proposition~\ref{prop:gen-eikonal1}, $w\in\lipl(\Omega)$ and is a Monge subsolution of \eqref{eikonal}. The Monge supersolution property of $w$ is a direct consequence of Proposition~\ref{prop:gen-eikonal2}.
\end{proof}

\begin{proof}[Proof of Theorem \ref{perron:gen_eikonal}]
Let us first show that $\ol{w}$ is a Monge solution of \eqref{eikonal}. It is clear from Proposition \ref{prop:gen-eikonal1} that \(\ol{w}\leq u_+\), \(\ol{w}\in\lipl(\Omega)\), and \(\ol{w}\) is a Monge subsolution of \eqref{eikonal}. It remains to prove that \(\ol{w}\) is also a Monge supersolution.

Suppose, by contradiction, that there exists \(x_0\in\Omega\) such that
\[ 
q:=|\nabla^-\ol{w}|(x_0)<f(x_0). 
\]
Let $p_0\in (q, f(x_0))$. By the continuity of $f$, one can take $r_0>0$ sufficiently small such that $\ol{B_{r_0}(x_0)}\subset\Omega$ and $p_0<\inf_{B_{r_0}(x_0)}f$. Moreover, for any $\eta\in (0, p_0-q)$, by letting $r_0>0$ further small, we have 
\[
\ol{w}(x_0)\leq \ol{w}(y)+(p_0-\eta)d(x_0,y) \quad \text{for all  $y\in B_{r_0}(x_0)$.}
\]
If $\ol{w}(x_0)=u_+(x_0)$, then using the fact that $\ol{w}\leq u_+$ in $\Omega$, we have
\[ 
u_+(x_0)=\ol{w}(x_0)\leq \ol{w}(y)+(p_0-\eta)d(x_0,y)\leq u_+(y)+(p_0-\eta)d(x_0,y) 
\]
for all $y\in B_{r_0}(x_0)$. This immediately implies $|\nabla^- u_+|(x_0)\leq p_0-\eta<f(x_0)$, which is a contradiction to the supersolution property of $u_+$. 

Let us now focus on the case when $\ol{w}(x_0)<u_+(x_0)$. By the continuity of $\ol{w}$ and $u_+$, we can choose $r\in (0, r_0)$ small such that 
\[
\sup_{\ol{B_{r}(x_0)}}\ol{w}+\frac{\eta r}{2}\leq \inf_{\ol{B_{r}(x_0)}}u_+.
\]
Define 
\begin{equation}\label{perron eq4-1}
    \psi(x):=\ol{w}(x_0)+\frac{\eta r}{2}-p_0d(x,x_0)\qquad\text{on}\ \ol{B_{r}(x_0)}. 
\end{equation}
One can easily verify that $\psi\leq u_+$, $|\nabla^-\psi|\leq p_0<f(x_0)$ on $B_r(x_0)$, 
\begin{equation}\label{strict-bump}
    \psi(x_0)= \ol{w}(x_0) +\eta r/2>\ol{w}(x_0),
\end{equation}
and
\begin{equation}\label{perron eq4}
    \psi \leq \ol{w}-\frac{\eta r}{2} \quad \text{on $\partial B_r(x_0)$. }
\end{equation}
We then define
\begin{equation}\label{perron eq5}
    \tilde{w}(x)=\begin{cases}
    \ \max\{\ol{w}(x),\psi(x)\}& \quad \text{if }x\in B_{r}(x_0), \\
    \ \ol{w}(x) &\quad \text{if } x\in \Omega\backslash B_{r}(x_0).
    \end{cases}
\end{equation}
It is clear that $u_-\leq \tilde{w}\leq u_+$ in $\Omega$ and $\tilde{w}\in \lipl(\Omega)$. Our construction of $\tilde{w}$ yields $|\nabla^-\tilde{w}|\leq f$ in $B_r(x_0)$ and 
\begin{equation}\label{bump-key}
    |\nabla^-\tilde{w}|(x)\leq p_0<\inf_{B_{r}(x_0)}f\quad \text{when $\tilde{w}(x)>\ol{w}(x)$.}
\end{equation}
Also, by \eqref{perron eq4-1} and \eqref{perron eq4}, we have $\tilde{w}=\ol{w}$ in $\Omega\setminus B_{r'}(x_0)$ for some $r'<r$, and therefore $|\nabla^-\tilde{w}|\leq f$ in $\Omega$. In other words, $\tilde{w}$ is a Monge subsolution of \eqref{eikonal}. In view of \eqref{strict-bump}, this contradicts the maximality of $\ol{w}$. Hence, $\ol{w}$ must be a Monge supersolution of \eqref{eikonal}. 

We next prove that $\ul{w}$ given in \eqref{Perron_sol_min} is also a Monge solution of \eqref{eikonal}. Define
\begin{equation}\label{sol-class}
    \begin{aligned}
    \ol{\S}&:=\left\{v\in \lipl(\Omega): v \text{ is a Monge supersolution of \eqref{eikonal} and } u_-\leq v\leq u_+ \text{ in $\Omega$}\right\}\\
    \S&:=\left\{v\in \lipl(\Omega): v \text{ is a Monge solution of \eqref{eikonal} and }
    u_-\leq v\leq u_+\text{ in $\Omega$}\right\}.
    \end{aligned}
\end{equation}
Since $\ol{w}$ is a Monge solution of \eqref{eikonal} such that $u_-\leq \ol{w}\leq u_+$ in $\Omega$, we have $\S\neq\emptyset$. We now claim that 
\begin{equation}\label{perron eq6}
    \ul{w}(x)=\inf_{v\in\ol{\S}}v(x)
    =\inf_{v\in\S}v(x)
    \qquad\text{for every }x\in\Om.
\end{equation}
Indeed, it is clear that $\S\subset\ol{\S}$ and therefore
\begin{equation}\label{perron eq7}
    \inf_{v\in\ol{\S}}v(x)
    \le \inf_{v\in\S}v(x).
\end{equation}
Conversely, for any fixed $v_0\in\ol{\S}$, applying our Perron-type result for $\ol{w}$ in \eqref{Perron_sol_max} with $u_+$ replaced by $v_0$, we obtain a Monge solution $w_0$ of \eqref{eikonal} satisfying $u_-\le w_0\le v_0\le u_+$ in $\Omega$. Hence, $w_0\in\S$, and consequently
\[
\inf_{v\in\S} v(x)\le w_0(x)\le v_0(x).
\]
Taking the infimum over all such $v_0\in\ol{\S}$ yields
\begin{equation}\label{perron eq8}
    \inf_{v\in\S}v(x)
    \le \inf_{v\in\ol{\S}}v(x).
\end{equation}
Combining \eqref{perron eq7} and \eqref{perron eq8}, we complete the proof of the claim  \eqref{perron eq6}.

Adopting Theorem~\ref{thm:eikonal_gen_sol} for $\inf_{v\in \S} v$, we see that $\ul{w}$ is a Monge solution of \eqref{eikonal}. By definition, it is also clear that $\ul{w}\leq \ol{w}$ in $\Omega$. Our proof is now complete.
\end{proof}

\subsection{More general Hamilton-Jacobi equations}

For more general Hamilton-Jacobi equations in the form of \eqref{general eq}, recall the definitions of Monge solutions (resp., Monge subsolutions, Monge supersolutions) in \eqref{Monge general}. Our strategy is to convert the equation to the eikonal type under (H1) along with the following conditions on $H\in C(\X\times\R\times[0,\infty))$. 
\begin{enumerate}
    \item[(H2)] For every $R>0$,
    \begin{equation*} 
        \inf_{(x,r)\in\X\times[-R,R]} H(x,r,p)\to \infty\quad \text{as $p\to\infty$}.
    \end{equation*}
    \item[(H3)] For any $(x,p)\in\X\times[0,\infty)$ and $r_1, r_2\in \R$,
    \begin{equation*}
        H(x,r_1,p)\geq H(x,r_2,p)\quad \text{if $ r_1\geq r_2$}. 
    \end{equation*}
\end{enumerate}
One may restrict the domain of $H$ as well as the assumptions (H1)--(H3) to $\Omega\times\R\times[0,\infty)$ without affecting our results below. Define
\begin{equation}\label{def-h}
    h(x, r)=\inf\{p\geq 0: H(x, r, p)\geq 0\} \quad \text{for $x\in \Omega$ and $r\in \R$.}
\end{equation}
Note that the condition (H2) implies $h(x, r)<\infty$ for any $(x, r)\in \Omega\times\R$. In addition, by (H3), we deduce 
\begin{equation}\label{h_mono}
    r\mapsto h(x, r) \quad \text{is nonincreasing in $\R$ for each $x\in \Omega$.}
\end{equation}
By (H1) and the continuity of $H$, we also obtain the continuity of $h$. 
Moreover, if for each $(x,r)\in\X\times\R$ there exists $p\geq 0$ satisfying $H(x,r,p)=0$, then such $p$ is unique and $h(x, r)=p$. It is also clear that $h(x, r)=0$ if $H(x, r, 0)\geq 0$. Using this function $h$, we can formally rewrite \eqref{general eq} as an eikonal equation in the form 
\begin{equation}\label{general loc_u}
    |\nabla u|(x)=h(x, u(x))\quad \text{for $x\in \Omega$}.
\end{equation} 
One can prove the following equivalence between these two formulations. 

\begin{lem}[Equivalent formulations]\label{lem implicit}
    Suppose that $(\X, d)$ is a complete length space and $\Omega\subsetneq\X$ is a domain. Assume that $H\in C(\X\times \R\times [0, \infty))$ satisfies (H1)(H2). Let $h$ be defined as in \eqref{def-h}. Then, the results below hold. 
    \begin{itemize}
        \item[(i)] $u$ is a Monge supersolution of \eqref{general eq}  if and only if $u$ is a Monge supersolution of \eqref{general loc_u}.
        \item[(ii)] If $u$ is a Monge subsolution of \eqref{general eq}, then $u$ is a Monge subsolution of \eqref{general loc_u}.
        \item[(iii)] Assume in addition that $h(x, u(x))>0$ for all $x\in \Omega$. If  $u$ is a Monge subsolution of \eqref{general loc_u}, then $u$ is a Monge subsolution of \eqref{general eq}.
    \end{itemize}
    In particular, when $h(x, u(x))>0$ for all $x\in \Omega$, $u$ is a Monge solution of \eqref{general eq}  if and only if $u$ is a Monge solution of \eqref{general loc_u}.
\end{lem}

\begin{proof}
(i) By the continuity of $H$ as well as (H1), we see from the definition of $h$ in \eqref{def-h} that for $x\in\Omega$, $r\in \R$ and $p\geq 0$, $H(x, r, p)\geq 0$ if and only if $p\geq h(x, r)$. Letting $r=u(x), p=|\nabla^- u|(x)$ with $u\in \lipl(\Omega)$ and $x\in \Omega$ arbitrary yields the equivalence between the Monge supersolution properties of \eqref{general eq} and \eqref{general loc_u}.

(ii) By (H1) and \eqref{def-h}, it is easily seen that for any $x\in \Omega$, $u\in \lipl(\Omega)$ satisfies $|\nabla^- u|(x)\leq h(x, u(x))$ provided that $H(x, u(x), |\nabla^- u|(x))\leq 0$. The conclusion (ii) then follows immediately. 

(iii) Under the assumption that $h(x, u(x))>0$ for any $x\in \Omega$, we can use (H1), (H2) and the continuity of $H$ to show that $H(x, u(x), h(x,u(x)))=0$. By (H1) again, we see that $|\nabla^- u|(x)\leq h(x, u(x))$ implies 
\[
H(x, u(x), |\nabla^- u|(x))\leq H(x, u(x), h(x,u(x)))=0.
\]
The proof of (iii) is thus complete. 

Clearly, the final assertion is a consequence of (i), (ii), and (iii). 
\end{proof}

Based on the form \eqref{general loc_u}, we shall show that the Monge subsolution property for \eqref{general eq} is preserved under a pointwise supremum. Let us first consider the case of two functions.

\begin{prop}\label{max_min_2_subs}
    Suppose that $(\X, d)$ is a complete length space and $\Omega\subsetneq\X$ is a domain. Assume that $H\in C(\X\times \R\times [0, \infty))$ satisfies (H1). Let $u_1$ and $u_2$ be Monge subsolutions of \eqref{general eq}. Then, $u,v:\Omega\to\R$ defined respectively by
    \[ 
    u(x):=\max\{u_1(x),u_2(x)\}, \quad v(x):=\min\{u_1(x),u_2(x)\}\quad \text{for $x\in\Omega$}, 
    \]
    are also Monge subsolutions of \eqref{general eq}.
\end{prop}

\begin{proof}
It is straightforward to verify that $u,v\in\lipl(\Omega)$. To show that $u,v$ are Monge subsolutions of \eqref{general eq}, we fix $x\in\Omega$ arbitrarily. Without loss of generality, suppose that $u(x)=u_1(x)$.
Since $u_1$ is a Monge subsolution of \eqref{general eq} and 
\[
|\nabla^-u|(x)=\limsup_{y\to x}
\frac{(u(x)-u(y))_+}{d(x,y)}\leq
\limsup_{y\to x}
\frac{(u_1(x)-u_1(y))_+}{d(x,y)}
=|\nabla^-u_1|(x),
\]
by (H1), we have
\[
H\bigl(x,u(x),|\nabla^-u|(x)\bigr)\leq
H\bigl(x,u_1(x),|\nabla^-u_1|(x)\bigr)
\leq 0.
\]
Hence, $u$ is a Monge subsolution of \eqref{general eq}, thanks to the arbitrariness of $x\in \Omega$.

The proof for $v$ is slightly more involved. Suppose that $u_1(x)<u_2(x)$. Then by continuity $u_1<u_2$ in a neighborhood of $x$, and thus $v=u_1$ near $x$. It follows that $|\nabla^-v|(x)=|\nabla^-u_1|(x)$ and
\[
H(x,v(x),|\nabla^-v|(x))
=H(x,u_1(x),|\nabla^-u_1|(x))\leq0.
\]
The case $u_2(x)<u_1(x)$ can be handled in a symmetric manner. If $u_1(x)=u_2(x)$, then
\[
|\nabla^-v|(x)=\max\{|\nabla^-u_1|(x),|\nabla^-u_2|(x)\}.
\]
Without loss of generality, we may assume that $|\nabla^-v|(x)=|\nabla^-u_1|(x)$. Since $v(x)=u_1(x)=u_2(x)$, we obtain
\[
H(x,v(x),|\nabla^-v|(x))
=H(x,u_1(x),|\nabla^-u_1|(x))\leq 0,
\]
which, by the arbitrariness of $x\in \Omega$, implies that $v$ is a Monge subsolution of \eqref{general eq}. 
\end{proof}

Using Lemma~\ref{lem implicit} and Proposition~\ref{max_min_2_subs}, we can now prove the stability of the subsolution property for \eqref{general eq} under a pointwise supremum and infimum.

\begin{prop}[Pointwise extrema of subsolutions to HJ equations]\label{stab_gen_subs}
    Suppose that $(\X, d)$ is a complete length space and $\Omega\subsetneq\X$ is a domain. Assume that $H\in C(\X\times\R\times[0,\infty))$ satisfies (H1)--(H3). Let $\mathcal{U}$ be a nonempty set of Monge subsolutions to \eqref{general eq}. 
    Suppose that $w_1$ and $w_2$ satisfy \eqref{point-sup} and \eqref{point-inf} respectively. Then, $w_1$ and $w_2$ are Monge subsolutions of \eqref{general eq}. 
\end{prop}

\begin{proof}
We first show the subsolution property for $w_1$. For any fixed $u_0\in \mathcal{U}$, let 
\[
\mathcal{U}_0:=\{\max\{u_0,u\}:\; u\in \mathcal{U}\}.
\]
In view of Proposition \ref{max_min_2_subs}, $\mathcal{U}_0$ is a set of Monge subsolutions to \eqref{general eq} satisfying
\[
\sup\{u(x):u\in\mathcal{U}_0\}=w_1(x) \quad\text{for all}\  x\in\Omega. 
\] 
By Lemma~\ref{lem implicit}(ii), each $u\in\mathcal{U}_0$ satisfies  $|\nabla^-u|(x)\leq h(x,u(x))$ for any $x\in\Omega$. Moreover, noticing that $r\mapsto h(x, r)$ is nonincreasing as stated in \eqref{h_mono}, we have 
\[ 
|\nabla^-u|(x)\leq h(x,u(x))\leq h(x,u_0(x)). 
\]
Therefore, $u$ is a Monge subsolution of \eqref{eikonal} with $f\in C(\Omega)$ given by $f(x)=h(x, u_0(x))$ for $x\in \Omega$. Note that $f$ is continuous and locally bounded in $\Omega$. Applying Proposition \ref{prop:gen-eikonal1} with $\Omega'=B_\rho(x)\subset\Omega$ for each $x\in \Omega$ and $\rho>0$ small, we prove that $w_1$ is a Monge subsolution of \eqref{eikonal} with the same $f$, that is, $w_1\in \lipl(\Omega')$ and         
\[
|\nabla^-w_1|(y)\leq h(y,u_0(y))\quad\text{for all $y\in \Omega'$. }
\]
Since $u_0\in \U$ and $x\in\Omega$ were arbitrary, by \eqref{point-sup} and the continuity of $h$, we deduce that $w_1\in \lipl(\Omega)$ and
\[ 
|\nabla^-w_1|(x)\leq h(x,w_1(x))\quad \text{for all}\ x\in\Omega. 
\]        
If $|\nabla^-w_1|(x)>0$, by Lemma~\ref{lem implicit}(iii) we have
\[ 
H(x,w_1(x),|\nabla^-w_1|(x))\leq 0.
\]
Now suppose that $|\nabla^-w_1|(x)=0$. In this case, we take $\{u_j\}\subset\mathcal{U}$ such that $u_j(x)\to w_1(x)$ as $j\to\infty$.  By (H1), we have 
\[ 
H(x,u_j(x),0)\leq H(x,u_j(x), |\nabla^- u_j|(x))\leq 0. 
\]
Passing to the limit as $j\to\infty$, by the continuity of $H$, we thus deduce that
\[ 
H(x,w_1(x),|\nabla^-w_1|(x))=H(x,w_1(x),0)\leq 0. 
\]
Combining both cases, we show that $w_1$ is a Monge subsolution of \eqref{general eq}.
    
Now we turn to the proof for $w_2$. Again, by the monotonicity of $r\mapsto h(x,r)$, we have
\[ 
|\nabla^-u|(x)\leq h(x,u(x))\leq h(x,w_2(x))\quad \text{for all $x\in\Omega$.}
\]
Fix $x_0\in\Omega$ arbitrarily. Since $w_2(x_0)>-\infty$, let $b<w_2(x_0)$. By the continuity of $h(\cdot, b)$, there exist $M, R>0$ such that $\ol{B_R(x_0)}\subset\Omega$ and $h(x, b)\leq M$ in $B_R(x_0)$. We claim that, 
\begin{equation}\label{eq:subs_pt1}
    u> b\quad\text{in $B_r(x_0)$ for every $u\in\mathcal{U}$}
\end{equation}
when $r>0$ is sufficiently small, satisfying 
\begin{equation}\label{small radius}
    r< \frac{w_2(x_0)-b}{M}. 
\end{equation}
Indeed, for each $u\in\mathcal{U}$, let $u_b:=\max\{u, b\}$. Then, for all $x\in B_R(x_0)$,
\[ 
|\nabla^- u_b|(x)=
\begin{cases}
    \ |\nabla^-u|(x)\leq h(x, b)\leq M,& \text{if $u(x)>b$,} \\
    \ 0,& \text{if $u(x)\leq b$,} 
\end{cases} 
\]
By Theorem~\ref{thm:char}(i), applied to $u_b/M$, we obtain the $M$-Lipschitz continuity of $u_b$ in $B_r(x_0)$ for $r>0$ small that satisfies \eqref{small radius}. Since $u(x_0)\geq w_2(x_0)$, this yields, for any $y\in B_r(x_0)$,  
\[
u_b(y)\geq u_b(x_0)-Mr> u_b(x_0)+b-w_2(x_0)\geq  b.
\]
The claim \eqref{eq:subs_pt1} immediately follows from the definition of $u_b$. Consequently, 
\[
|\nabla^-u|(x)\leq h(x, u(x))\leq h(x, b) \quad \text{for all $x\in B_r(x_0)$.} 
\]
By Proposition~\ref{prop:gen-eikonal1} with $f(x)=h(x, b)$ on $B_r(x)$, we obtain $w_2\in\lipl(B_r(x_0))$. Applying again Proposition~\ref{prop:gen-eikonal1} with $f(x)=h(x, w_2(x))$ yields  
\[ 
|\nabla^-w_2|(x)\leq h(x,w_2(x))\quad \text{for all $x\in B_r(x_0)$.} 
\]
To show $H(x, w_2(x), |\nabla^-w_2|(x))\leq 0$ for any $x\in B_r(x_0)$, one can follow the same argument for $w_1$, discussing the case $|\nabla^- w_2|(x)>0$ and the case $|\nabla^- w_2|(x)=0$ separately. Our proof is now complete, since $x_0\in \Omega$ is taken arbitrarily.  
\end{proof}

\begin{prop}[Pointwise minimum between supersolutions to HJ equations]\label{min_2_supers}
    Suppose that $(\X, d)$ is a complete length space and $\Omega\subsetneq\X$ is a domain. Assume that $H\in C(\X\times \R\times [0, \infty))$ satisfies (H1). Let $v_1$ and $v_2$ be Monge supersolutions of \eqref{general eq}. Then, $v:\Omega\to\R$ given by
    \[ 
    v(x):=\min\{v_1(x),v_2(x)\}\quad \text{for each}\ x\in\Omega, 
    \]
    is a Monge supersolution of \eqref{general eq}.
\end{prop}

\begin{proof}
It is clear that $v\in\lipl(\Omega)$. It remains to show that $v$ is a Monge supersolution of \eqref{general eq}. Fix $x\in\Omega$. Without loss of generality, suppose that $v(x)=v_1(x)$.
Since $v_1$ is a Monge supersolution of \eqref{general eq}, we have
\begin{equation}\label{min_super_eq_1}
    H\bigl(x,v_1(x),|\nabla^-v_1|(x)\bigr)\geq 0.
\end{equation}
Noticing that
\[
|\nabla^-v|(x)=\limsup_{y\to x}
\frac{(v(x)-v(y))_+}{d(x,y)}\geq
\limsup_{y\to x}
\frac{(v_1(x)-v_1(y))_+}{d(x,y)}
=|\nabla^-v_1|(x), 
\]
by \eqref{min_super_eq_1} and (H1), we obtain
\[
H\bigl(x,v(x),|\nabla^-v|(x)\bigr)
\geq H\bigl(x,v_1(x),|\nabla^-v_1|(x)\bigr)\geq 0.
\]
Hence, $v$ is a Monge supersolution of \eqref{general eq}, due to the arbitrariness of $x\in \Omega$.
\end{proof}

\begin{prop}[Pointwise infimum of supersolutions to HJ equations]\label{stab:super_gen_HJ}
    Suppose that $(\X, d)$ is a complete length space and $\Omega\subsetneq\X$ is a domain. Assume that $H\in C(\X\times\R\times[0,\infty))$ satisfies (H1)--(H3). Let $\mathcal{U}$ be a nonempty set of Monge supersolutions to \eqref{general eq}. Suppose that  \begin{equation}\label{inf_supers_gen}
        w_3(x):=\inf\{u(x): u\in\mathcal{U}\}>-\infty \quad \text{for each $x\in \Omega$.}
    \end{equation}
    If $w_3\in\lipl(\Omega)$, then $w_3$ is a Monge supersolution of \eqref{general eq}.
\end{prop}

\begin{proof}
For any fixed $u_0\in \mathcal{U}$, let $\mathcal{U}_0:=\{\min\{u_0,u\}:\; u\in \mathcal{U}\}$. In view of Proposition \ref{min_2_supers}, $\mathcal{U}_0$ is still a set of Monge supersolutions to \eqref{general eq}. It is also clear that 
\[
\inf\{u(x): u\in\mathcal{U}_0\}=w_3(x) \quad\text{for all}\  x\in\Omega. 
\]
By Lemma~\ref{lem implicit}(i), each $u\in\mathcal{U}_0$ satisfies 
$|\nabla^-u|(x)\geq h(x,u(x))$ for any $x\in\Omega$. Moreover, noticing that $r\mapsto h(x, r)$ is nonincreasing as stated in \eqref{h_mono}, we have 
\[ 
|\nabla^-u|(x)\geq h(x,u(x))\geq h(x,u_0(x)). 
\]
Therefore, $u$ is a Monge supersolution of \eqref{eikonal} with $f\in C(\Omega)$ given by $f(x)=h(x, u_0(x))$ for $x\in \Omega$. Note that $f$ is continuous and nonnegative in $\Omega$. 
Applying Proposition~\ref{prop:gen-eikonal2} and the assumption that $w_3\in\lipl(\Omega)$, we prove that $w_3$ is a Monge supersolution of \eqref{eikonal} with such $f$. Since $u_0\in \U$ is taken arbitrarily, by \eqref{inf_supers_gen} and the continuity of $h$, we deduce that $|\nabla^-w_3|(x)\geq h(x,w_3(x))$ for each $x\in\Omega$. It follows from Lemma~\ref{lem implicit}(i) again that $w_3$ is a Monge supersolution of \eqref{general eq}.
\end{proof}

\begin{thm}[Pointwise infimum of solutions to HJ equations]\label{cor:HJ_sol}
    Suppose that $(\X, d)$ is a complete length space and $\Omega\subsetneq\X$ is a domain. Assume that $H\in C(\X\times\R\times[0,\infty))$ satisfies (H1)--(H3). Let $\mathcal{U}$ be a nonempty collection of Monge solutions to \eqref{general eq}. Suppose that 
    \begin{equation*}
        w(x):=\inf\{u(x): u\in\mathcal{U}\}>-\infty \quad \text{for each $x\in \Omega$.}
    \end{equation*}
    Then, $w$ is a Monge solution of \eqref{general eq}.
\end{thm}

\begin{proof}
By Proposition~\ref{stab_gen_subs}, $w\in\lipl(\Omega)$ is a Monge subsolution of \eqref{general eq}. Since each $u\in\mathcal{U}$ is a Monge supersolution of \eqref{general eq}, it follows from Proposition~\ref{stab:super_gen_HJ} that $w$ is also a Monge supersolution of \eqref{general eq}. Hence, $w$ is a Monge solution of \eqref{general eq}. 
\end{proof}

We now proceed to establish a Perron-type existence result for Monge solutions of \eqref{general eq}. 

\begin{thm}[Perron's method for HJ equations]\label{gen_perron's method}
    Suppose that $(\X, d)$ is a complete length space and $\Omega\subsetneq\X$ is a domain. Assume that $H\in C(\X\times\R\times[0,\infty))$ and (H1)--(H3) hold. Assume that there exist a Monge subsolution $u_-$ and a Monge supersolution $u_+$ of \eqref{general eq} such that $u_-\leq u_+$ in $\Omega$. Define $\overline{w},\ul{w}:\Omega\to\R$ by  
    \begin{equation}\label{perron-def gen_max}
        \ol{w}(x):=\sup\{u(x): \text{$u$ is a Monge subsolution of \eqref{general eq} satisfying $u_-\leq u\leq u_+$ in $\Omega$}\}, 
    \end{equation}
    \begin{equation*}
        \ul{w}(x):=\inf\{v(x): \text{$v$ is a Monge supersolution of \eqref{general eq} satisfying $u_-\leq v\leq u_+$ in $\Omega$}\}. 
    \end{equation*}
    Then, $\ol{w}$ and $\ul{w}$ are Monge solutions of \eqref{general eq} with $\ul{w}\leq \ol{w}$.
\end{thm}

\begin{proof}
Let us first verify the solution property for $\ol{w}$. It is clear that $\ol{w}\leq u_+$ on $\Omega$. In view of Proposition~\ref{stab_gen_subs}, $\ol{w}$ is a Monge subsolution of \eqref{general eq}, which by Lemma~\ref{lem implicit}(ii) yields $|\nabla^- \overline{w}|\leq f$ in $\Omega$ with $f\in C(\Omega)$ defined by  
\begin{equation}\label{h-to-f}
    f(x)=h(x, \ol{w}(x)), \quad x\in \Omega.
\end{equation}
It suffices to show that $\ol{w}$ is a Monge supersolution of \eqref{general eq}. Assume by contradiction that $\ol{w}$ is not a Monge supersolution, i.e., there exists $x_0\in\Omega$ such that 
\[ 
H(x_0,\overline{w}(x_0),|\nabla^-\overline{w}|(x_0))<0. 
\]
Then, we have $\ol{w}(x_0)<u_+(x_0)$, since $u_+$ is a Monge supersolution. Moreover, by the continuity of $H$ and the definition of $h$, we obtain $p_0>0$ such that 
\[
|\nabla^- \ol{w}|(x_0)<p_0< h(x_0, \ol{w}(x_0))=f(x_0).
\]
Following the proof of Theorem \ref{perron:gen_eikonal} for $f$ given by \eqref{h-to-f}, we can construct $\tilde{w}$ as in \eqref{perron eq5} with $\psi$ still defined by \eqref{perron eq4-1} and $\eta, r>0$ taken sufficiently small. Note that for $x\in \Omega$, 
\begin{equation}\label{perron-gen eq1}
    H(x, \tilde{w}(x), |\nabla^- \tilde{w}|(x))\leq 0
\end{equation}
holds if $\tilde{w}(x)=\ol{w}(x)$. If on the other hand $\tilde{w}(x)>\ol{w}(x)$, we have $x\in B_r(x_0)$ and $\tilde{w}=\psi$ in a neighborhood of $x$, so that $|\nabla^-\tilde{w}|(x)\leq p_0$. Moreover, on $B_r(x_0)$, 
\[
|\tilde{w}(x)-\ol{w}(x_0)|\leq \sup_{y\in B_r(x_0)}|\ol{w}(y)-\ol{w}(x_0)|+\dfrac{\eta r}{2}.
\]
The right-hand side tends to zero as $r\to 0$. Since $H(x_0,\ol{w}(x_0),p_0)<0$, the continuity of $H$ allows us to choose $r$ small that $H(x,\tilde{w}(x),p_0)<0$ throughout $B_r(x_0)$. By (H1), \eqref{perron-gen eq1} holds at such $x$ as well. We conclude that $\tilde{w}$ is a Monge subsolution of \eqref{general eq} satisfying $u_-\leq \tilde{w}\leq u_+$ in $\Omega$ and $\tilde{w}(x_0)>\ol{w}(x_0)$. This is a contradiction to \eqref{perron-def gen_max}, and therefore $\ol{w}$ is a Monge solution of \eqref{general eq}.

As for $\ul{w}$, the same argument as in Theorem~\ref{perron:gen_eikonal} applies. Introducing the supersolution class $\ol{\S}$ and the solution class $\S$ for \eqref{general eq} analogously to \eqref{sol-class}, we can obtain \eqref{perron eq6} again in this general setting. Applying Theorem~\ref{cor:HJ_sol}, we conclude that $\ul{w}$ is a Monge solution of \eqref{general eq}.
\end{proof}

\bibliographystyle{abbrv}

\begin{thebibliography}{10}

\bibitem{AF}
L.~Ambrosio and J.~Feng.
\newblock On a class of first order {H}amilton-{J}acobi equations in metric spaces.
\newblock {\em J. Diff. Equ.}, 256(7):2194--2245, 2014.

\bibitem{BCBook}
M.~Bardi and I.~Capuzzo-Dolcetta.
\newblock {\em Optimal control and viscosity solutions of
  {H}amilton-{J}acobi-{B}ellman equations}.
\newblock Systems \& Control: Foundations \& Applications. Birkh{\"a}user
  Boston Inc., Boston, MA, 1997.
\newblock With appendices by Maurizio Falcone and Pierpaolo Soravia.

\bibitem{BrDa}
A.~Briani and A.~Davini.
\newblock Monge solutions for discontinuous {H}amiltonians.
\newblock {\em ESAIM Control Optim. Calc. Var.}, 11(2):229--251, 2005.

\bibitem{CCM}
F.~Camilli, R.~Capitanelli, and C.~Marchi.
\newblock Eikonal equations on the {S}ierpinski gasket.
\newblock {\em Math. Ann.}, 364(3-4):1167--1188, 2016.


\bibitem{CaSi1}
F.~Camilli and A.~Siconolfi.
\newblock Hamilton-{J}acobi equations with measurable dependence on the state variable.
\newblock {\em Adv. Diff. Equ.}, 8(6):733--768, 2003.

\bibitem{CIL}
M.~G.~Crandall, H.~Ishii, and P.-L.~Lions.
\newblock User's guide to viscosity solutions of second order partial
  differential equations.
\newblock {\em Bull. Amer. Math. Soc. (N.S.)}, 27(1):1--67, 1992.

\bibitem{DaLeSa}
A.~Daniilidis, T.~M.~Le, and D.~Salas.
\newblock Metric compatibility and determination in complete metric spaces.
\newblock {\em Math. Z.}, 308(4):Paper No. 62, 31, 2024.

\bibitem{DaSa}
A.~Daniilidis and D.~Salas.
\newblock A determination theorem in terms of the metric slope.
\newblock {\em Proc. Amer. Math. Soc.}, 150(10):4325--4333, 2022.

\bibitem{E1}
I.~Ekeland. 
\newblock On the variational principle. 
\newblock {\em J. Math. Anal. Appl.}, 47:324–353, 1974.

\bibitem{E2}
I.~Ekeland. 
\newblock Nonconvex minimization problems. 
\newblock {\em Bull. Amer. Math. Soc. (N.S.)}, 1(3):443–474, 1979.

\bibitem{EGV}
F.~Essebei, G.~Giovannardi, and S.~Verzellesi.
\newblock Monge solutions for discontinuous Hamilton-Jacobi equations in Carnot groups.
\newblock {\em NoDEA Nonlinear Diff. Equ. Appl.},
31, Paper No. 95, 2024.

\bibitem{GaS2}
W.~Gangbo and A.~{\'S}wi{\polhk{e}}ch.
\newblock Optimal transport and large number of particles.
\newblock {\em Discrete Contin. Dyn. Syst.}, 34(4):1397--1441, 2014.

\bibitem{GaS}
W.~Gangbo and A.~{\'S}wi{\polhk{e}}ch.
\newblock Metric viscosity solutions of {H}amilton-{J}acobi equations depending
  on local slopes.
\newblock {\em Calc. Var. Part. Diff. Equ.}, 54(1):1183--1218, 2015.

\bibitem{GHN}
Y.~Giga, N.~Hamamuki, and A.~Nakayasu.
\newblock Eikonal equations in metric spaces.
\newblock {\em Trans. Amer. Math. Soc.}, 367(1):49--66, 2015.

\bibitem{JC}
H.~Jiang and X.~Cui.
\newblock Metric viscosity solutions and distance-like functions on the Wasserstein space.
\newblock {\em Diﬀer. Geom. Appl.}, 103, Paper No. 102387, 2026.

\bibitem{LiS}
P.-L.~Lions and P.~E.~Souganidis.
\newblock Differential games, optimal control and directional derivatives of viscosity solutions of Bellman's and Isaacs' equations.
\newblock {\em SIAM J. Control Optim.}, 23(4):566--583, 1985. 

\bibitem{LMit1}
Q.~Liu and A.~Mitsuishi.
\newblock Principal eigenvalue problem for infinity {L}aplacian in metric
  spaces.
\newblock {\em Adv. Nonlinear Stud.}, 22(1):548--573, 2022.

\bibitem{LN}
Q.~Liu and A.~Nakayasu. Convexity preserving properties for Hamilton-Jacobi equations in geodesic spaces.
\newblock {\em Discrete Contin. Dyn. Syst.}, 39(1):157--183, 2019.

\bibitem{LShZ}
Q.~Liu, N.~Shanmugalingam, and X.~Zhou.
\newblock Equivalence of solutions of eikonal equation in metric spaces.
\newblock {\em J. Diff. Equ.}, 272:979--1014, 2021.

\bibitem{LShZ2}
Q.~Liu, N.~Shanmugalingam, and X.~Zhou.
\newblock Discontinuous eikonal equations in metric measure spaces.
\newblock {\em Trans. Amer. Math. Soc.}, 378(1):695--729, 2025.

\bibitem{LW}
Q.~Liu and M.~B.~P. Wiranata.
\newblock Monge solutions of time-dependent Hamilton-Jacobi equations in metric spaces.
\newblock {\em ESAIM Control Optim. Calc. Var.}, 32, Paper No. 2, 39 pp., 2026.

\bibitem{LTG}
T.~M.~L\^e and S.~Tapia-Garc\'ia.
\newblock On (discounted) global eikonal equations in metric spaces,
\newblock to appear in {\em Ann. Sc. Norm. Super. Pisa Cl. Sci. (5).} 

\bibitem{M}
S.~Makida.
\newblock On the Gromov--Hausdorff stability of metric viscosity solutions.
\newblock {\em Proc. Amer. Math. Soc.}, 154(4):1701--1714, 2026.

\bibitem{MN}
S.~Makida and A.~Nakayasu.
\newblock Stability of metric viscosity solutions under Hausdorff convergence, 
\newblock to appear in {\em Anal. Geom. Metr. Spaces}.

\bibitem{Na1}
A.~Nakayasu.
\newblock Metric viscosity solutions for {H}amilton-{J}acobi equations of evolution type.
\newblock {\em Adv. Math. Sci. Appl.}, 24(2):333--351, 2014.

\bibitem{Na}
A.~Nakayasu.
\newblock Homogenization of Hamilton-Jacobi equations on the Sierpinski gasket.
\newblock Talk at NLPDE seminar, Kyoto University, October 5, 2021. 

\bibitem{NN}
A.~Nakayasu and T.~Namba.
\newblock Stability properties and large time behavior of viscosity solutions of Hamilton-Jacobi equations on metric spaces.
\newblock {\em Nonlinearity}, 31(11):5147--5161, 2018.

\bibitem{NeSu}
R.~T. Newcomb, II and J.~Su.
\newblock Eikonal equations with discontinuities.
\newblock {\em Diff. Integral Equ.}, 8(8):1947--1960, 1995.

\bibitem{So}
P.~Soravia. 
\newblock Boundary value problems for Hamilton-Jacobi equations with discontinuous Lagrangian. 
\newblock {\em Indiana Univ. Math. J.}, 51(2):451--477, 2002.

\bibitem{S}
A.~{\'S}wi{\polhk{e}}ch.
\newblock Risk-sensitive control and differential games in infinite dimensions.
\newblock {\em Nonlinear Anal.}, 50(4):509--522, 2002. 

\bibitem{TrBook}
H.~V.~Tran. 
\newblock {\it Hamilton-Jacobi equations---theory and applications}. 
\newblock Graduate Studies in Mathematics, 213, Amer. Math. Soc., Providence, RI, 2021.

\bibitem{TrYu}
H.~V.~Tran and Y.~Yu. 
\newblock Optimal convergence rate for periodic homogenization of convex Hamilton-Jacobi equations.
\newblock {\em Indiana Univ. Math. J.}, 74(3):555--573, 2025. 

\end{thebibliography}

\end{document}